\documentclass[10pt, a4paper]{amsart}

\usepackage{xcolor}
\usepackage{fancyhdr}
\usepackage{tikz}
\usepackage[vmargin=100pt, hmargin=60pt]{geometry}
\usetikzlibrary{decorations.pathmorphing}
\usepackage[T1]{fontenc}
\usepackage{ucs}
\usepackage[utf8]{inputenc}
\usepackage{amsfonts}
\usepackage{listings}
\usepackage{amsmath}
\usepackage[mathscr]{euscript}
\usepackage{amssymb}
\usepackage{amsfonts}
\usepackage{dsfont}
\usepackage{amsthm}
\usepackage{fancybox}
\usepackage{enumitem}
\usepackage{dsfont}
\usepackage{algpseudocode}
\usepackage{algorithm}
\usepackage{graphicx}
\usepackage{caption}
\usepackage{subcaption}
\usepackage{float}
\usepackage{authblk}
\usepackage{relsize}
\usepackage{setspace}
\usepackage{mathtools}
\usepackage{comment}

\usepackage{fancybox} 
\usepackage{pifont} 
\usepackage[utf8]{inputenc}
\usepackage[T1]{fontenc}
\usepackage{appendix}
\usepackage{lmodern}
\usepackage{newunicodechar}
\usepackage{amsmath, amsthm, amsfonts,amssymb}
\usepackage{xcolor,mhsetup}
\usepackage[extra,safe]{tipa}
\usepackage{mathtools}
\usepackage{geometry}
\usepackage{mathrsfs}
\usepackage{hyperref}
\usepackage{url}
\usepackage{fancyhdr}
\theoremstyle{definition}
\newtheorem{theorem}{Theorem}[section]
\newtheorem{prop}{Proposition}[section]
\newtheorem{algo}{Algorithm}[section]

\newtheorem{lem}{Lemma}[section]

\theoremstyle{remark}
\newtheorem{remark}[theorem]{Remark}
\theoremstyle{cor}

\numberwithin{equation}{section}

\def\my_c{c_\infty}

\newcommand{\mynewtheorem}[2]{
  \newaliascnt{#1}{dummy}
  \newtheorem{#1}[#1]{#2}
  \aliascntresetthe{#1}
  \expandafter\def\csname #1autorefname\endcsname{#2}
}

\newcommand{\be}{\begin{equation}}
\newcommand{\ee}{\end{equation}}
\newcommand{\bde}{\begin{displaymath}}
\newcommand{\ede}{\end{displaymath}}
\newcommand{\beq}{\begin{eqnarray*}}
\newcommand{\eeq}{\end{eqnarray*}}
\newcommand{\beqa}{\begin{eqnarray}}
\newcommand{\eeqa}{\end{eqnarray}}
\newcommand{\bel }{\left\{\begin{array}{ll}}
\newcommand{\eel}{\cr \end{array} \right.}

\newcommand{\dcb}{\begin{array}{lll}}
\newcommand{\dce}{\end{array}}
\newcommand{\ebe}{\begin{enumerate}\setlength{\baselineskip}{13pt}\setlength{\parskip}{0pt}}
\newcommand{\dbe}{\end{enumerate}}

\def\0{{\mathbf{0}}}

\def\d{{{\rm d}}}

\usepackage{mwe}

\begin{document}

\setcounter{tocdepth}{1}

\title{Deterministic computation of quantiles in a Lipschitz framework}

\maketitle

\begin{center}
Yurun Gu
\footnote{address: SAMOVAR, Telecom SudParis, Institut Polytechnique de Paris, 19 rue Marguerite Perey, 91120 Palaiseau, France \\
 e-mail: yurun.gu@telecom-sudparis.eu},
Cl\'ement Rey 
\footnote{address: CMAP, École Polytechnique, Institut Polytechnique de Paris, Route de Saclay, 91120 Palaiseau, France\\
 e-mail: clement.rey@polytechnique.edu \\
This work was supported by the Chair Stress Test, RISK Management and Financial Steering, led by the French Ecole Polytechnique and its Foundation and sponsored by BNP Paribas.}
\end{center}

%



\begin{abstract}
In this article, we focus on computing the quantiles of a random variable $f(X)$, where $X$ is a $[0,1]^d$-valued random variable, $d \in \mathbb{N}^{\ast}$, and $f:[0,1]^d\to \mathbb{R}$ is a deterministic Lipschitz function. We are particularly interested in scenarios where the cost of a single function evaluation is high, while the law of $X$ is assumed to be known. In this context, we propose a deterministic algorithm to compute deterministic lower and upper bounds for the quantile of $f(X)$ at a given level $\alpha \in (0,1)$. With a fixed budget of $N$ function calls, we demonstrate that our algorithm achieves an exponential deterministic convergence rate for $d=1$ ($\mathcal{O}( \rho^N)$ with $\rho \in (0,1)$) and a polynomial deterministic convergence rate for $d>1$ ($\mathcal{O}(N^{-\frac{1}{d-1}})$) and show the optimality of those rates. Furthermore, we design two algorithms, depending on whether the Lipschitz constant of $f$ is known or unknown.
\end{abstract}


\tableofcontents

\section{Introduction}\label{sec:introduction}

In this article, we consider a $\Omega:=[0,1]^{d}$-valued ($d \in \mathbb{N}^{\ast}$) random variable $X$ and a Lipschitz function $f:[0,1]^{d}\to \mathbb{R}$. We are interested in scenarios where the law of $X$ is known and our focus lies on the number of calls to $f$. This approach is relevant in situations where $f$ incurs high computational cost, while $X$ is a random variable with well-known distribution, such as the uniform distribution. We propose an algorithm that uses $N \in \mathbb{N}^{\ast}$ calls to $f$ in order to compute an approximation of the $\alpha$-order quantile of $f(X)$, $\alpha \in (0,1)$, defined by
\begin{align}
\label{def:quantile}
q_{\alpha}(f,X)  :=& \inf \{l \in \mathbb{R}, \mathbb{P}(f(X) \leqslant l) \geqslant \alpha \}.
\end{align}
Our contribution is to provide, for any $N \in \mathbb{N}^{\ast}$, a deterministic approximation $q^{N}_{\alpha}(f,X)$ for $q_{\alpha}(f,X)$ and to demonstrate that (see Theorem \ref{th:conv_gen} and Theorem \ref{th:conv_gen_unknown}), the approximation error converges to zero as $N$ tends to infinity. More generally, we derive a deterministic rate of convergence for any sufficiently large $N$. This rate has exponential nature for $d=1$ ($\vert q^{N}_{\alpha}(f,X)-q_{\alpha}(f,X) \vert \leqslant C \rho^{N}$ with $\rho \in (0,1)$ and $C>0$) and has polynomial nature when $d>1$ ($\vert q^{N}_{\alpha}(f,X)-q_{\alpha}(f,X) \vert \leqslant  CN^{-\frac{1}{d-1}}$ with $C>0$). \\

Quantile computation has a wide range of applications across various domains, including banking  \cite{Edgeworth_1888}, database optimization  \cite{Greenwald_Khanna_2016}, sociology 
\cite{Ullah_Luo_Adebayo_Sunday_Kartal_2023}
, binary classification \cite{Tambwekar_Maiya_Dhavala_Saha_2021} or sensor networks \cite{Lee_epedelenlioglu_Spania_Muniraju_2020}. 
Additionally, related challenges such as multivariate quantiles \cite{Hallin_Paindaveine_Siman_2010} or estimating the difference of quantiles at two given levels \cite{Yang_Zhao_2018} also warrant attention. It is worth mentioning that, many existing quantile estimation methods are derived from the equivalent problems of minimizing loss functions \cite{Koenker_2005}, 
\begin{align*}
    q_{\alpha}(f,X)=\underset{y \in \mathbb{R}} {\mbox{argmin}}\{\alpha E[(f(X)-y)_+]+(1-\alpha)E[(f(X)-y)_-]\},
\end{align*}
in which we denote $a_{+} :=\max\{a,0\}$ and $a_{-} :=\max\{-a,0\}$ for $a \in \mathbb{R}$. For some applications, we refer $e.g.$, to quantile estimation methods designed in \cite{Gomes_Pestana_2007} or \cite{Bellini_Klar_Muller_Rosazza_2014}. \\

When computing quantity of interest related to the law of $f(X)$, a common approach, when possible, is to use the simulation based Monte Carlo methods. These techniques involve generating a large sample of simulations (usually independent) of $f(X)$ to approximate our quantity of interest. In our context, an extensive literature is dedicated to performing quantile computation through Monte Carlo methods. Under mild conditions, the Monte Carlo estimator of $q_{\alpha}(f,X)$ typically satisfies a central limit theorem with a weak convergence rate given by $CN^{-\frac{1}{2}}$. Many endeavours aim to reduce the variance of these estimators, represented by the constant $C$. The importance sampling approach, initially showcased in \cite{Glynn_2011}, has since been developed further in subsequent works such as \cite{Goffinet_Wallach_1996}, \cite{Egloff_Leippold_2010}, and \cite{Byon_Ko_Lam_Pan_Young_2020}. Similarly, the design of control variates can be adapted, as demonstrated in \cite{Hsu_Nelson_1990}, \cite{Hesterberg_Nelson_1998}, and \cite{Cannamela_Garnier_Iooss_2008}. Additionally, multi-level splitting methods, as described in \cite{DelMoral_Furon_Guyader_2012}, \cite{Cerou_Guyader_2016}, and \cite{Cerou_Guyader_Rousset_2019}, have been applied to quantile computation in \cite{Guyader_Hengartner_Matzner_2011}.\\

Considering not only the nature or the value of the rate of convergence, our algorithm differs from a standard Monte Carlo approach in the following way: Our algorithm is deterministic, meaning it does not rely on simulations of random variables with form $f(X)$. It simply calls the function $f$ at some specific points which are selected sequentially according to the law of $X$ and to previous calls to $f$. We thus manage to derive deterministic rates of convergence, whereas Monte Carlo methods provide weak rates due to the convergence in law of the central limit theorem. Hence, we provide deterministic upper and lower bounds for $q_{\alpha}(f,X)$ while a Monte Carlo approach provides lower and upper bounds by the way of a confidence interval.  In particular, with Monte Carlo methods, there is a small, albeit nonzero, probability that $q_{\alpha}(f,X)$ lies outside the computed confidence interval. Moreover, with Monte Carlo methods, the confidence interval is asymptotic in the limit $N \to + \infty$, whereas our deterministic intervals are given for any finite $N$ larger than an explicit constant.\\

We design a first algorithm that requires evaluating $f$ at specific points, as well as knowing its Lipschitz constant for implementation, while a Monte Carlo approach only requires the ability to simulate $f(X)$. However, we also introduce a second algorithm, which requires that $f$ is Lipschitz but does not need to have access to the Lipschitz constant. These algorithms are inspired by the methodologies outlined in \cite{Cohen_Devore_Petrova_Wojtaszczyk_2014}, originally developed to approximate the minimum of a Lipschitz function. In this context, as in our case, the rate of convergence depends on the dimension of $\Omega$, with exponential convergence when $d=1$ and polynomial convergence when $d>1$. It is worth noting that these algorithms were further adapted in \cite{Bernard_Cohen_Guyader_Malrieu_2022} for the computation of failure probabilities, $i.e.$, for calculating $\mathbb{P}(f(X)>c)$ for a given $c \in \mathbb{R}$, when the Lipschitz constant of $f$ is known. In this case, similar convergence regimes are observed depending on the value of $d$.\\


Our approach employs a dichotomous strategy, leveraging the Lipschitz property of $f$ to systematically exclude certain regions $\mathcal{R}$ within $[0,1]^{d}$ where we know $q_{\alpha}(f,X) \notin f(\mathcal{R})$. By avoiding the evaluation of $f$ in these regions, we achieve convergence with the expected rate. When the Lipschitz constant is unknown, we adopt similar but parallel computations. Each of them mimics the known Lipschitz constant algorithm, but using exponentially increasing Lipschitz constant candidates. The number of calls to $f$ allocated to each candidate decreases as the candidate's value increases. Remarkably, this yields convergence rates of comparable order to the scenarios where the Lipschitz constants are known.\\

Furthermore, we establish the optimality of the deterministic rates achieved by the algorithms we develop. In essence, under our framework, we show that any algorithm providing an approximation $\tilde{q}^{N}_{\alpha}(f,X)$ for $q_{\alpha}(f,X)$ ($i.e.$, $\tilde{q}^{N}_{\alpha}(f,X)$ is built by calling the function $f$ at $N$ points) converges, for at least one triplet $(f,X,\alpha)$ in our framework, with deterministic rate not faster than exponential, $\vert \tilde{q}^{N}_{\alpha}(f,X)-q_{\alpha}(f,X) \vert \geqslant  \tilde{C} \tilde{\rho}^{N}$,  $\tilde{\rho} \in (0,1)$, when $d=1$, and polynomial, $\vert \tilde{q}^{N}_{\alpha}(f,X)-q_{\alpha}(f,X) \vert \geqslant  \tilde{C}N^{\frac{1}{d-1}}$, $\tilde{C}>0$, when $d>1$.\\

The article is presented in the following way. The algorithm for the computation of the approximation $q^{N}_{\alpha}(f,X)$ for the quantile of order $\alpha$ of $f(X)$ when the Lipschitz constant of $f$ is known is exposed in Section \ref{sec:Lips_known}, Algorithm \ref{algo:gen_Lip_connu}. Its convergence, along with expected rates, is addressed in Theorem \ref{th:conv_gen}.  In the case of unknown Lipschitz constant, the algorithm is presented in Section \ref{sec:Lip_unknown}, Algorithm \ref{algo:gen_Lip_inconnu} and the related convergence result is presented in Theorem \ref{th:conv_gen_unknown}.The optimality of the rates of convergence achieved by Algorithms \ref{algo:gen_Lip_connu} and \ref{algo:gen_Lip_inconnu} is proved in Section \ref{sec:opt}., Propositions \ref{prop:optimality_d>1} and \ref{prop:optimality_d=1}. A numerical illustration of the convergence of both algorithms is given in Section \ref{Numerical_illustration}.

\section{Computation of quantile with known Lipschitz constant}\label{sec:Lip_unknown}
\label{sec:Lips_known}
In this section, we design an algorithm for the computation of $q_{\alpha}(f,X)$ when the Lipschitz constant of $f$ is known. We also prove convergence of this algorithm towards $q_{\alpha}(f,X)$ with explicit rate (polynomial or exponential regarding the value of $d$) in Theorem \ref{th:conv_gen}. This convergence is established under the following assumptions concerning $f,X$ and $\alpha$. Notice that the exact same assumptions will be used in the next Section when the Lipschitz constant is unknown.
\begin{enumerate}

\item \label{hyp:lipschitz}
$f:\mathbb{R}^{d} \to \mathbb{R}$ is a Lipschitz function: There exists $ L_f >0$ such that
 \begin{align*}
 \forall x,y \in [0,1]^d, \qquad \vert f(x) - f(y) \vert \leqslant L_f \vert x - y \vert_{\mathbb{R}^{d}},
 \end{align*}
 where $\vert \cdot \vert_{\mathbb{R}^{d}}$ is the usual Euclidean norm in $\mathbb{R}^{d}$.
 \vspace{3mm}
\item \label{hyp:level_set} Level set assumption : There exists $M>0$ such that
\begin{align*}
\forall \delta >0, \quad \lambda_{\mbox{Leb}}(x \in [0,1]^{d}, f(x) \in [q_{\alpha}(f,X)-\delta,q_{\alpha}(f,X)+\delta])\leqslant M \delta,
\end{align*}
where $\lambda_{\mbox{Leb}}$ is the usual Lebesgue measure.
\end{enumerate}

\subsection{Preliminaries}
We begin by introducing standard elements and results that will be involved in the definition and the proof of convergence of our algorithm. Our first step is to subdivide the set $[0,1]^{d}$. We begin by introducing the center points of our subdivisions. For $k \in \mathbb{N}$, the level of the subdivision, and $i \in \{0,\ldots,3^{k}-1  \}$, we define
\begin{align*}
d_{i}^{k}= \frac{2i+1}{2 \times 3^{k}},
\end{align*}
and for $\beta \in \{0,\ldots,3^{k}-1  \}^{d}$, we focus on the following subdivision,
\begin{align*}
D_{\beta}^{k}= \{ x \in [0,1]^{d}, x_{j} \in [d_{\beta_{j}}^{k} - \frac{1}{2 \times 3^{k}} , d_{\beta_{j}}^{k} + \frac{1}{2 \times 3^{k}}  ) \cup \{1\}^{\mathds{1}_{\beta_{j}=3^{k}-1}} \},
\end{align*}
with notation $\{1\}^{0} =\emptyset$ and $\{1\}^{1}=\{1\}$. Notice that the subdivisions are distinct and $\cup_{\beta \in \{0,\ldots,3^{k}-1  \}^{d}} D_{\beta}^{k}=[0,1]^{d} $. We also define $\delta_{\beta}^{k} := \sup_{x \in D_{\beta}^{k}} \vert x- d_{\beta}^{k} \vert $ ($d_{\beta}^{k} = (d_{\beta_{1}}^{k} ,\ldots , d_{\beta_{d}}^{k} )$) and remark that $\delta_{\beta}^{k}$ does not depend on $\beta$, so we simply denote $\delta^{k}$ in the sequel. In particular, we have
\begin{align*}
\delta^{k} =& (\sum_{i=1}^{d} (d_{\beta_{i}}^{k}-(d_{\beta_{i}}^{k}+\frac{1}{2 \times 3^{k}}))^{2})^{\frac{1}{2}} =  \frac{d^{\frac{1}{2}}}{2} \frac{1}{3^{k}}.
\end{align*}

We are now in a position to introduce our approximation for $q_{\alpha}(f,X)$. For $k \in \mathbb{N}$,  we introduce 
\begin{align}
\label{def:approx_general}
\mathfrak{q}_{\alpha}^{k}(f,X)=\sup\{f(d_{\beta}^{k}) ,\beta \in \{0,\ldots,3^{k}-1\}^d, \sum_{\gamma \in \{0,\ldots,3^{k}-1\}^d,  f(d_{\gamma}^{k})\geqslant f(d_{\beta}^{k})} \mathbb{P}(X\in D^{k}_{\gamma}) \geqslant 1- \alpha \},
\end{align}
with $\sup \emptyset := \inf_{\beta \in \{0,\ldots,3^{k}-1\}^d}f(d_{\beta}^{k})$. It happens that $\mathfrak{q}_{\alpha}^{k}(f,X)$ is the value returned by our algorithm if it has enough budget to reach level of subdivision $k$ (and not higher). When we want to emphasize that the computation depends on the Lipschitz constant considered, we will denote $\mathfrak{q}_{L_{f},\alpha}^{k}(f,X) $ instead of $\mathfrak{q}_{\alpha}^{k}(f,X)$. \\

In definition (\ref{def:approx_general}), the computation is made on all $\beta \in \{0,\ldots,3^{k}-1\}^d$, which implies high computational cost if $\mathfrak{q}_{\alpha}^{k}(f,X)$ is used naively as an approximation for $q(f,X)$. We tackle this issue by introducing a sequence $\Pi^{k}\subset \{0,\ldots,3^{k}-1\}^d$, $k \in \mathbb{N}$, which is built recursively and such that, at step $k$, new computations of $f(d_{\beta}^{k})$ are only made for $\beta \in \Pi^{k}$. Exploiting the Lipschitz property, we will show in Lemma \ref{lem:def_approx} that, we can replace $ \{0,\ldots,3^{k}-1\}^d$ by $\Pi^{k}$ and $f(d_{\gamma}^{k})$ by a well choosen value when $\gamma \notin \Pi^{k}$ in  (\ref{def:approx_general}), without changing the definition of $\mathfrak{q}_{\alpha}^{k}(f,X)$. We now introduce this alternative way of defining $\mathfrak{q}_{\alpha}^{k}(f,X)$.\\

Let us define $\Pi^{0}=\{(0,\ldots,0)\}=\{ \{ 0 \}^{d}\}$, and $\tilde{\mathfrak{q}}^{0}_{\alpha}(f,X)=f(d^{0}_{ \{ 0 \}^{d}})$. Let us assume that for $k \in \mathbb{N}$, we have access to $\Pi^{k}$.  We compute  $f^{k}_{\beta}:=f(d^{k}_{\beta})$, $\beta \in \Pi^{k}\subset \{0,\ldots,3^{k}-1\}^d$ and if $\beta \notin   \Pi^k$, no new computation is done and we choose $f^{k}_{\beta} \in \{f(d_{\beta}^{k}) \} \cup \{f(d_{\beta^{-[l]}}^{k-l}),l \in\{1,\ldots,k\},\beta^{-[l-1]}\notin \Pi^{k-l+1} \}$ where for $\beta \in \{0,\ldots,3^{k}-1\}^d$, $\beta^{-[0]}=\beta$ and for $k \in \mathbb{N}^{\ast}$, $\beta^{-[1]}=   \gamma \in \{0,\ldots,3^{k-1}-1\}^d$ such that $\gamma_{i} =\lfloor \frac{ \beta_{i}}{3} \rfloor, i \in \{1,\ldots,d\}, $ and $\beta^{-[l+1]}=(\beta^{-[l]})^{-[1]}, l \in \mathbb{N},l \leqslant k$.\\

We define recursively
\begin{align} 
\label{def:approx_general_restr}
\tilde{\mathfrak{q}}^{k}_{\alpha}(f,X)=\sup\{f(d_{\beta}^{k}),\beta \in \Pi^{k}, \sum_{\gamma \in \{0,\ldots,3^{k}-1\}^d,  f_{\gamma}^{k} \geqslant f(d_{\beta}^{k})} \mathbb{P}(X\in D^{k}_{\gamma}) \geqslant 1- \alpha \},
\end{align}
and
\begin{align}
\label{def:Pik}
\Pi^{k+1}=\{3 \beta+\{0,1,2\}^d,\beta \in \Pi^k , f(d^{k}_{\beta}) \in [\tilde{\mathfrak{q}}^{k}_{\alpha}- 2L_f \delta^{k},\tilde{\mathfrak{q}}^{k}_{\alpha}+ 2 L_f \delta^{k}] \} .
\end{align}
We will sometimes denote $\Pi_{L_{f}}^{k}:=\Pi^{k}$ and $\tilde{\mathfrak{q}}^{k}_{L_{f},\alpha}(f,X):=\tilde{\mathfrak{q}}^{k}_{\alpha}(f,X)$ to emphasize the fact that $\Pi^{k}$ is computed in (\ref{def:Pik}) using $L_{f}$ as a Lipschitz constant. In particular such a notation may be relevant when we use an upper bound for $L_{f}$ instead of its true value. Notice that is  if we choose $L \geqslant L' \geqslant L_{f}$, we have $ \Pi_{L_{f}}^{k} \subset \Pi_{L'}^{k} \subset \Pi_{L}^{k} $.


Due to the flexibility in the choice of $f^{k}_{\gamma}$, one might argue that $\tilde{\mathfrak{q}}^{k}_{\alpha}(f,X)$ may differ $w.r.t.$ the value chosen for $f^{k}_{\gamma}$. In the following result, we show that all the definitions are actually equivalent. We also prove that $\tilde{\mathfrak{q}}^{k}_{L,\alpha}(f,X)=\tilde{\mathfrak{q}}^{k}_{L_{f},\alpha}(f,X)$ for $L \geqslant L_{f}$.
\begin{lem}
\label{lem:def_approx}
For every $k \in \mathbb{N}$, the following assertions hold true.
\begin{enumerate}[label=\textbf{\Alph*.}]
\item \label{lem:def_approx_point1} 
The quantity $\mathfrak{q}_{\alpha}^{k}(f,X)$ defined in (\ref{def:approx_general}) satisfies
\begin{align}
\label{def:approx_general_bis}
\mathfrak{q}_{\alpha}^{k}(f,X)=\inf\{f(d_{\beta}^{k}) ,\beta \in \{0,\ldots,3^{k}-1\}^d, \sum_{\gamma \in \{0,\ldots,3^{k}-1\}^d,  f(d_{\gamma}^{k})\leqslant f(d_{\beta}^{k})} \mathbb{P}(X\in D^{k}_{\gamma}) \geqslant  \alpha \}.
\end{align}
\item \label{lem:def_approx_point2}
The quantity $\tilde{\mathfrak{q}}^{k}_{\alpha}(f,X)$ defined in (\ref{def:approx_general_restr}) satisfies
\begin{align*}
\tilde{\mathfrak{q}}^{k}_{\alpha}(f,X)=\inf\{f(d_{\beta}^{k}),\beta \in \Pi^{k}, \sum_{\gamma \in \{0,\ldots,3^{k}-1\}^d,  f_{\gamma}^{k} \leqslant f(d_{\beta}^{k})} \mathbb{P}(X\in D^{k}_{\gamma}) \geqslant \alpha \}.
\end{align*}
\item \label{lem:def_approx_point3}
Moreover,
\begin{align*}
\mathfrak{q}^{k}_{\alpha}(f,X)= \tilde{\mathfrak{q}}^{k}_{\alpha}(f,X).
\end{align*}
\end{enumerate}
\end{lem}

\begin{proof}
Let us prove \ref{lem:def_approx_point1}. The result is immediate for $k=0$. We fix $k \in \mathbb{N}^{\ast}$ and introduce
\begin{align*}
\mathfrak{p}_{\alpha}^{k}:=\inf\{f(d_{\beta}^{k}) ,\beta \in \{0,\ldots,3^{k}-1\}^d, \sum_{\gamma \in \{0,\ldots,3^{k}-1\}^d,  f(d_{\gamma}^{k})\leqslant f(d_{\beta}^{k})} \mathbb{P}(X\in D^{k}_{\gamma}) \geqslant  \alpha \} .
\end{align*}
We remark that (denoting shortly $\mathfrak{q}_{\alpha}^{k}:=\mathfrak{q}_{\alpha}^{k}(f,X)$),
\begin{align*}
\mathbb{P}(X \in \bigcup_{\gamma \in \{0,\ldots,3^{k}-1\}^d,f(d_{\gamma}^{k}) < \mathfrak{q}_{\alpha}^{k} } D_{\gamma}^{k}) < \alpha.
\end{align*}
Thus, since the subdivisions $D_{\gamma}^{k},\gamma \in \{0,\ldots,3^{k}-1\}^d$, are distinct, then $\mathfrak{p}_{\alpha}^{k} \geqslant \mathfrak{q}_{\alpha}^{k}$. Moreover,  by definition of $\mathfrak{q}_{\alpha}^{k}$,
\begin{align*}
\mathbb{P}(X \in \bigcup_{\gamma \in \{0,\ldots,3^{k}-1\}^d,f(d_{\gamma}^{k}) > \mathfrak{q}_{\alpha}^{k} } D_{\gamma}^{k}) < 1-\alpha,
\end{align*}
and it follows that
\begin{align*}
\mathbb{P}(X \in  \bigcup_{\gamma \in \{0,\ldots,3^{k}-1\}^d,f(d_{\gamma}^{k}) \leqslant \mathfrak{q}_{\alpha}^{k} } D_{\gamma}^{k}) \geqslant \alpha.
\end{align*}
Since there exists $\beta^{\ast} \in \{0,\ldots,3^{k}-1\}^d$, such that $f(d_{\beta^{\ast}}^{k}) = \mathfrak{q}_{\alpha}^{k} $,  we conclude that $\mathfrak{q}_{\alpha}^{k}  \geqslant \mathfrak{p}_{\alpha}^{k} $, and the proof of \ref{lem:def_approx_point1} is completed.  \\

Now let us prove \ref{lem:def_approx_point3}. The result is true for $k=0$ and let us prove it for $k \in \mathbb{N}^{\ast}$. Using definition (\ref{def:approx_general}), we have $q_{\alpha}(f,X) \geqslant \mathfrak{q}_{\alpha}^{k} - L_{f} \delta^{k}$, and using the relationship (\ref{def:approx_general_bis}), we have  $q_{\alpha}(f,X) \leqslant \mathfrak{q}_{\alpha}^{k} + L_{f} \delta^{k}$. Assume that there exists $\beta_{0} \notin \Pi^{k}$ such that $\mathfrak{q}_{\alpha}^{k}=f(d^{k}_{\beta_{0}})$. Using that $\beta_{0} \notin \Pi^{k}$, we have
\begin{align*}
\vert \mathfrak{q}_{\alpha}^{k-1}-f(d^{k-1}_{\beta_{0}^{-[1]}}) \vert > 2 L_{f} \delta^{k-1}.
\end{align*}
Hence,
\begin{align*}
\vert q_{\alpha}(f,X)-f(d^{k}_{\beta_{0}}) \vert  \geqslant &\vert \mathfrak{q}_{\alpha}^{k-1}-f(d^{k-1}_{\beta_{0}^{-[1]}}) \vert  - \vert f(d^{k-1}_{\beta_{0}^{-[1]}}) -f(d^{k}_{\beta_{0}}) \vert - \vert \mathfrak{q}_{\alpha}^{k-1}-q_{\alpha}(f,X) \vert \\
> & 2 L_{f} \delta^{k-1} -2 L_{f} \delta^{k}-L_{f} \delta^{k-1} \\
=&  L_{f} \delta^{k-1} -2 L_{f} \delta^{k}=L_{f} \delta^{k} ,
\end{align*}
which is impossible due to the previous observation. In particular, we can replace the condition $\beta \in \{0,\ldots,3^{k}-1\}^d$ by $\beta \in \Pi^{k}$ in (\ref{def:approx_general}). It remains to observe that $f(d_{\gamma}^{k})$ may be replaced by $f^{k}_{\gamma}$ as soon as $\gamma \notin \Pi^{k}$. Indeed, for such $\gamma$, using the same calculus as above, $f(d^{k-1}_{\gamma^{-[1]}})- \mathfrak{q}_{\alpha}^{k-1}$ and $f(d^{k}_{\gamma})- \mathfrak{q}_{\alpha}^{k}$ necessarily have the same sign, and the proof of \ref{lem:def_approx_point3} is completed. Using similar arguments, we obtain the proof of \ref{lem:def_approx_point2}.
\end{proof}

Furthermore, we can obtain an upper bound for the error between $q_{\alpha}(f,X)$ and our estimator, regarding that our algorithm reaches level $k$ (and not higher), $i.e.$, returns $ \mathfrak{q}_{\alpha}^{k}$.
\begin{lem}
\label{lem:intervalle_confiance}
For any $k \in \mathbb{N}$, we have
\begin{align}
\label{eq:intervalle_confiance}
\vert q_{\alpha}(f,X)- \mathfrak{q}_{\alpha}^{k} \vert \leqslant L_{f} \delta^{k}.
\end{align}

\end{lem}

\begin{proof}
This result follows directly from Lemma \ref{lem:def_approx} \ref{lem:def_approx_point1}. By applying (\ref{def:approx_general}), and noting that the subdivisions $D_{\beta}^{k},\beta \in \{0,\ldots,3^{k}-1\}^d$, are distinct, we obtain $q_{\alpha}(f,X) \geqslant \mathfrak{q}_{\alpha}^{k} - L_{f} \delta^{k}$. Additionally, using the relationship (\ref{def:approx_general_bis}), we have $q_{\alpha}(f,X) \leqslant \mathfrak{q}_{\alpha}^{k} + L_{f} \delta^{k}$.

\end{proof}

\subsection{Algorithm and Main result}
In this section, we introduce Algorithm \ref{algo:gen_Lip_connu} which is applicable when the Lipschitz constant $L_{f}$ is known and establish its convergence in Theorem \ref{th:conv_gen}. To build Algorithm \ref{algo:gen_Lip_connu},  we leverage Lemma \ref{lem:def_approx} and return $\mathfrak{q}^{k}_{\alpha}(f,X)$ for $k$, the highest subdivision level achievable with a global budget $N$. It is important to note that the knowledge of the constant $M$, which appears in the level-set assumption (\ref{hyp:level_set}), is not necessary to implement the algorithm.

\begin{algo}
\label{algo:gen_Lip_connu}  \quad \\
\noindent \textbf{Fix} $N \in \mathbb{N}^{\ast}$ the maximum number of calls to $f$ \\
$\Pi^{0}=\{(0,\ldots,0)\}$.\\
$k=0$. \\
$N_{\mbox{call}}=1$.\\
\textbf{While} $N_{\mbox{call}}\leqslant N$\\
\vspace{2mm}
\quad Compute $f(d^{k}_{\beta})$, $\beta \in \Pi^k $ (When $d^{k}_{\beta}=d^{k-1}_{\beta^{-[1]}}$, the computation was already done before step $k$).\\
\vspace{2mm}
\quad\textbf{Set} 
$f^{k}_{\beta}=f(d^{k}_{\beta})$, for every $\beta \in \Pi^k $ and $f^{k}_{\beta}=f^{k-1}_{\beta^{-[1]}}$, for every $\beta \notin \Pi^k $ .\\
\vspace{2mm}
\quad\textbf{Set} 
\begin{align*}
q^{N}_{\alpha}(f,X)=&\sup\{f(d_{\beta}^{k}) ,\beta \in \Pi^{k}, \sum_{\gamma \in \{0,\ldots,3^{k}-1\}^d,  f^{k}_{\gamma}\geqslant f(d_{\beta}^{k})} \mathbb{P}(X\in D^{k}_{\gamma}) \geqslant 1- \alpha \} \\
=&\inf\{f(d_{\beta}^{k}),\beta \in \Pi^{k}, \sum_{\gamma \in \{0,\ldots,3^{k}-1\}^d,  f_{\gamma}^{k} \leqslant f(d_{\beta}^{k})} \mathbb{P}(X\in D^{k}_{\gamma}) \geqslant \alpha \},
\end{align*}
\qquad and $\underline{q}^{N}_{\alpha}(f,X)=q^{N}_{\alpha}(f,X)-L_{f} \delta^{k}$,$\overline{q}^{N}_{\alpha}(f,X)=q^{N}_{\alpha}(f,X)+L_{f} \delta^{k}.$ \\
\vspace{2mm}
\quad \textbf{Set} $$\Pi^{k+1}=\{3 \beta+\{0,1,2\}^d,\beta \in \Pi^k , f(d^{k}_{\beta}) \in [\underline{q}^{N}_{\alpha}- L_f \delta^{k},\overline{q}^{N}_{\alpha}+ L_f \delta^{k}] \},$$
\qquad and $N_{\mbox{call}}=N_{\mbox{call}}+\frac{3^d-1}{3^d}\mbox{Card}(\Pi^{k+1})$.\\
\vspace{2mm}
\quad \textbf{Set}   \quad $k=k+1$.\\
\textbf{End While}\\
 \textbf{Return} $q^{N}_{\alpha}(f,X)$, $\underline{q}^{N}_{\alpha}(f,X)$, $\overline{q}^{N}_{\alpha}(f,X)$.
\end{algo}
We observe that to apply this algorithm, we need to have access to the law of $X$, or more specifically to $\mathbb{P}(X\in D^{k}_{\gamma})$, which is why this law is supposed to be known. When the probabilities $\mathbb{P}(X\in D^{k}_{\gamma})$ are unknown but $X$ may be simulated easily, one can use a Monte Carlo approach for their computation, which can be done once for all independently from the function $f$. We do not discuss further this approach in this paper. \\

The following result establishes the convergence of Algorithm \ref{algo:gen_Lip_connu} with explicit upper bounds for the errors exhibiting two regimes,  exponential or polynomial, regarding that $d=1$ or $d>1$.
\begin{theorem}
\label{th:conv_gen}
Let $N \in \mathbb{N}^{\ast}$, and let $q^{N}_{\alpha}(f,X)$, $\underline{q}^{N}_{\alpha}$ and $\overline{q}^{N}_{\alpha}$ be defined as in Algorithm \ref{algo:gen_Lip_connu}.  Assume that (\ref{hyp:lipschitz}) and (\ref{hyp:level_set}) hold. Then
\begin{align}
\label{eq:int_conf_algo_gen}
 q_{\alpha}(f,X) \in [ \underline{q}^{N}_{\alpha}(f,X),\overline{q}^{N}_{\alpha}(f,X)].
\end{align}

Moreover, if $d=1$, then
\begin{align}
\label{eq:borne_algo_gen_taille_intervalle_confiance_d_1}
\vert q_{\alpha}(f,X) -q^{N}_{\alpha}(f,X) \vert 
\leqslant C \rho^{N},
\end{align}
with $C= \frac{1}{2} L_{f} 3^{1+ \frac{1}{4 M L_{f}  }} $ and $\rho=3^{- \frac{1}{4 M L_{f}  }}.$ \\

If $d>1$ and $N>1$, then
\begin{align}
\label{eq:borne_algo_gen_taille_intervalle_confiance}
\vert q_{\alpha}(f,X) -q^{N}_{\alpha}(f,X) \vert 
\leqslant  C( N-1)^{-\frac{1}{d-1}},
\end{align}
with $C=  \frac{3}{2} L_{f} d^{\frac{1}{2}}  (3^{d}M L_{f} d^{\frac{1}{2}})^{\frac{1}{d-1}}$.
\end{theorem}

\begin{remark}
The above result can be extended to the case where the support of $f$ is not restricted to $[0,1]^{d}$ but to a Cartesian product of bounded intervals $\tilde{\Omega}:= [a_{1},b_{1}] \times \ldots \times [a_{d},b_{d}]$, $b_{i} > a_{i}$ for $i \in \{1,\ldots,d\}$. Indeed, if $X \in \tilde{\Omega}$, a rescaling can be used with the application of Algorithm \ref{algo:gen_Lip_connu} to approximate $q_{\alpha}(f,X)$, the quantile of order $\alpha$ of $f(X)$, for a Lipschitz function $f$ defined on $\tilde{\Omega}$. In particular, for every $x \in \tilde{\Omega}$, we can write $f(x)=g(h(x))$, with $h(x)_{i}=\frac{x_{i}-a_{i}}{b_{i}-a_{i}}$, by defining, for every $y \in [0,1]^{d}$, $g(y)=f(v(y))$ with $v(y)_{i}=a_{i}+(b_{i}-a_{i})y_{i}$. Consequently, the function $g$ is defined on $[0,1]^{d}$ and is Lipschitz with Lipschitz constant $L_{g} \leqslant c_{1} L_{f}$, with $c_{1}=\sup \{b_{i}-a_{i},i \in \{1,\ldots,d\} \}$. Moreover, if $f$ satisfies the level set assumption,
\begin{align*}
\forall \delta >0, \quad \lambda_{\mbox{Leb}}(x \in \tilde{\Omega}, f(x) \in [q_{\alpha}(f,X)-\delta,q_{\alpha}(f,X)+\delta])\leqslant M \delta, \quad M>0,
\end{align*}
then, using the change of variables formula, for every $\delta>0$, we have
\begin{align*}
\lambda_{\mbox{Leb}}(x \in [0,1]^{d}, g(x) \in & [q_{\alpha}(g,h(X))-\delta,q_{\alpha}(g,h(X))+\delta])\\
= & c_{2}\lambda_{\mbox{Leb}}(x \in \tilde{\Omega}, g(h(x)) \in [q_{\alpha}(f,X)-\delta,q_{\alpha}(f,X)+\delta]) \\
\leqslant &   c_{2}M  \delta,
\end{align*}

with $c_{2}= \prod_{i=1}^{d}(b_{i}-a_{i})$.  In conclusion $g,h(X)$ and $\alpha$ satisfy assumptions (\ref{hyp:lipschitz}) and (\ref{hyp:level_set}) and we can apply Algorithm \ref{algo:gen_Lip_connu} with the random variable $h(X)$ instead of $X$ and the function $g$ instead of $f$ to produce approximations for $q_{\alpha}(f,X)$. In particular
\begin{align*}
q_{\alpha}(f,X) = q_{\alpha}(g,h(X)) \in [\underline{q}^{N}_{\alpha}(g, h(X)),\overline{q}^{N}_{\alpha}(g,h (X))],
\end{align*}
and similar results as the those established in Theorem \ref{th:conv_gen} hold with $L_{f}$ replaced by $c_{1}L_{f}$ and $M$ replaced by $c_{2}M$ in the upper bounds.
\end{remark}

\subsection{Proof of convergence of the algorithm}
\label{sec:convergence}

Our strategy consists in identifying the maximum level $k(N)$ such that Algorithm \ref{algo:gen_Lip_connu} is sure to compute $\mathfrak{q}_{\alpha}^{k(N)}$ given an initial budget of $N$ calls to $f$. Then we can use Lemma \ref{lem:intervalle_confiance} with $k=k(N)$ and obtain the expected bounds. \\

To begin, for $k \in \mathbb{N}$, we introduce $N_{k}$ the number of calls to the function $f$ satisfying (\ref{hyp:lipschitz}), in order to compute $\mathfrak{q}_{\alpha}^{k}$.  Since each new subdivision of a set with form $D_{\beta}^{l}$ yields $3^{d}-1$ new calls to the function $f$, we have
\begin{align*}
N_{k} 
=& 1+(3^{d}-1)  \sum_{l=0}^{k-1} \sum_{\gamma \in \{0,\ldots,3^{l}-1\}^d}  \mathds{1}_{ f(d_{\gamma}^{l})\in [\mathfrak{q}_{\alpha}^{l} -2L_{f}  \delta^{l}  ,\mathfrak{q}_{\alpha}^{l}  +2L_{f}  \delta^{l}  ]} .
\end{align*}
In addition, we thus define $k(N):=\sup \{k \in \mathbb{N}, N_{k} \leqslant N\}$. Notice that $k(N)$ depends on $L_{f}$. When necessary, we will emphasize this dependence denoting $k(N,L_{f})$ instead of $k(N)$.\\

The next result establishes some upper bounds for $N_{k}$ with explicit dependence $w.r.t.$ $k$.
\begin{lem}
\label{lem:borne_Nk}
Let $k \in \mathbb{N}$. Assume that (\ref{hyp:lipschitz}) and (\ref{hyp:level_set}) hold. \\

\noindent If $d=1$, then
\begin{align*}
N_{k} \leqslant 1+4M L_{f} k.
\end{align*}
If $d>1$, then
\begin{align*}
N_{k} \leqslant 1+ {3^{d}}{2} M L_{f} d^{\frac{1}{2}}  \frac{3^{k(d-1)}-1}{3^{d-1}-1}.
\end{align*}
\end{lem}

\begin{proof}
First we remark that
\begin{align*}
N_{k} =& 1+(3^{d}-1)  \sum_{l=0}^{k-1} \sum_{\gamma \in \{0,\ldots,3^{l}-1\}^d}  \mathds{1}_{ f(d_{\gamma}^{l})\in [\mathfrak{q}_{\alpha}^{l} -2 L_{f}  \delta^{l}  ,\mathfrak{q}_{\alpha}^{l}  +2 L_{f}  \delta^{l}  ]} \\
= &1+ (3^{d}-1) \sum_{l=0}^{k-1} \sum_{\gamma \in \{0,\ldots,3^{l}-1\}^d}  3^{ld} \lambda_{\mbox{Leb}}(D^{l}_{\gamma}  )  \mathds{1}_{ f(d_{\gamma}^{l})\in [\mathfrak{q}_{\alpha}^{l} -2 L_{f}  \delta^{l}  ,\mathfrak{q}_{\alpha}^{l}  +2 L_{f}  \delta^{l}  ]}\\
\leqslant &1+ (3^{d}-1) \sum_{l=0}^{k-1} 3^{ld} \lambda_{\mbox{Leb}}(x \in [0,1]^{d}, f(x) \in [\mathfrak{q}_{\alpha}^{l} -3L_{f}  \delta^{l}  ,\mathfrak{q}_{\alpha}^{l}  +3L_{f}  \delta^{l}  ])   .
\end{align*}
Now, we introduce $r_{\alpha}^{l}=4L_{f}  \delta^{l}$. Then, using the level set assumption (\ref{hyp:level_set}) and that $[\mathfrak{q}_{\alpha}^{l} -3L_{f}  \delta^{l}  ,\mathfrak{q}_{\alpha}^{l}  +3L_{f}  \delta^{l}  ] \subset [q_{\alpha}(f,X)-r_{\alpha}^{l},q_{\alpha}(f,X)+r_{\alpha}^{l}])$ as a consequence of Lemma \ref{lem:intervalle_confiance}, we obtain

\begin{align*}
\lambda_{\mbox{Leb}}(x \in [0,1]^{d}, f(x) \in [\mathfrak{q}_{\alpha}^{l} -3L_{f}  \delta^{k}  ,\mathfrak{q}_{\alpha}^{l}  +3L_{f}  \delta^{l}  ]) \leqslant & M r_{\alpha}^{l},
\end{align*}
and then

\begin{align*}
N_{k} \leqslant &1+ (3^{d}-1) \sum_{l=0}^{k-1} M 3^{ld} r_{\alpha}^{l} \\
\leqslant &1+ (3^{d}-1)2 M L_{f} d^{\frac{1}{2}}  \sum_{l=0}^{k-1} 3^{l(d-1)} .
\end{align*}
We thus conclude easily depending on the value of $d$.

%
\end{proof}

We are now in a position to derive a lower bound for $k(N)$ that depends explicitly on $N$.
\begin{lem}
\label{lemme:borne_algo_taille_intervalle_confiance}
Let $N \in \mathbb{N}^{\ast}$. Assume that (\ref{hyp:lipschitz}) and (\ref{hyp:level_set}) hold. \\

\noindent If $d=1$, then
\begin{align*}
k(N) \geqslant \lfloor \frac{N-1}{4 M L_{f} } \rfloor ,
\end{align*}

and 

\begin{align}
\label{eq:borne_algo_taille_intervalle_confiance}
\vert q_{\alpha}(f,X)- \mathfrak{q}_{\alpha}^{k(N)} \vert \leqslant \frac{3}{2} L_{f} 3^{- \frac{N-1}{4 M L_{f} }}.
\end{align}

If $d>1$, then
\begin{align*}
k(N)  \geqslant   \lfloor \frac{ \ln(N-1)-\ln(3^{d} M L_{f} d^{\frac{1}{2}})}{(d-1) \ln(3) } \rfloor ,
\end{align*}

and 

\begin{align}
\label{eq:borne_algo_taille_intervalle_confiance_d>1}
\vert q_{\alpha}(f,X)- \mathfrak{q}_{\alpha}^{k(N)} \vert\leqslant \frac{3}{2} L_{f} d^{\frac{1}{2}}  ( \frac{ N-1}{3^{d}M L_{f} d^{\frac{1}{2}}} )^{-\frac{1}{d-1}} .
\end{align}

\end{lem}

\begin{proof}
The proof of the lower bound on $k(N)$ is a straightforward application of Lemma \ref{lem:borne_Nk} together with $\delta^{k} = \frac{d^{\frac{1}{2}}}{2} \frac{1}{3^{k}}$.  Then (\ref{eq:borne_algo_taille_intervalle_confiance}) and (\ref{eq:borne_algo_taille_intervalle_confiance_d>1}) are derived as direct applications of Lemma \ref{lem:intervalle_confiance}.
\end{proof}

\begin{proof}[Proof of Theorem \ref{th:conv_gen}].
We remark that $q_{\alpha}^{N}(f,X)  = \mathfrak{q}_{\alpha}^{k}$, $\underline{q}_{\alpha}^{N}(f,X)  = \mathfrak{q}_{\alpha}^{k}-L_{f} \delta^{k}  $ and that $\overline{q}_{\alpha}^{N} (f,X) =\mathfrak{q}_{\alpha}^{k}+L_{f} \delta^{k} $ for $k \geqslant k(N)$. Then the proof of (\ref{eq:int_conf_algo_gen}) is a consequence of Lemma \ref{lem:intervalle_confiance}. The proof \ref{eq:borne_algo_gen_taille_intervalle_confiance_d_1} and \ref{eq:borne_algo_gen_taille_intervalle_confiance} follows directly from Lemma \ref{lemme:borne_algo_taille_intervalle_confiance} (see (\ref{eq:borne_algo_taille_intervalle_confiance}) and (\ref{eq:borne_algo_taille_intervalle_confiance_d>1})). 
\end{proof}

\section{Computation of quantile with unknown Lipschitz constant}\label{sec:Lip_unknown}

In this Section, we design Algorithm \ref{algo:gen_Lip_inconnu} which is adapted to the computation of $q_{\alpha}(f,X)$ when the Lipschitz constant of $f$ is unknown and prove its convergence in Theorem \ref{th:conv_gen_unknown}. Similarly as when the Lipischitz constant is known, we observe two regimes (exponential or polynomial) for the error rate regarding the value of $d$.

\subsection{Preliminaries}
We suppose in this section that $L_{f}$ is unknown and for $j \in \mathbb{N}$ we introduce
\begin{align*}
L_{f}(j)=3^{j},
\end{align*}
as the candidates for the unknown Lipischitz constant $L_{f}$ of $f$. Let us consider a given global budget $N \in \mathbb{N}$, $N \geqslant \frac{\pi^{2}}{6}$, and introduce $N_{max}(j,N):=\lfloor \frac{6 N}{\pi^{2}(j+1)^{2} } \rfloor$ the respective budgets allocated for computations that mimic Algorithm \ref{algo:gen_Lip_connu} in the case the true Lipschitz constant is $L_{f}(j)$.  We also denote  $j_{max}(N)=\sup\{j \in \mathbb{N},N_{max}(j,N) \geqslant 1\}=\sup\{j \in \mathbb{N}, j \leqslant  \frac{\sqrt{6 N}}{\pi} -1\}=\lfloor \frac{\sqrt{6 N}}{\pi} \rfloor-1$.\\

Let us define $\Pi^{\diamond,0}(j)=\{(0,\ldots,0)\}=\{0^{d}\}$, and $\tilde{\mathfrak{q}}^{0}_{\alpha}(f,X)=f(d^{0}_{\{ 0 \}^{d}})$. Assume that for $k \in \mathbb{N}$,  we compute (if necessary) $f^{k}_{\beta}:=f(d^{k}_{\beta})$, $\beta \in  \cup_{j=1}^{j_{max}}  \Pi^{\diamond,k} (j)$ and if $\beta \notin   \cup_{j=1}^{j_{max}}  \Pi^{\diamond,k} (j)$, no new computation is done and we choose $f^{k}_{\beta} =f^{k-1}_{\beta^{-[1]}}$.\\

We now define recursively
\begin{align} 
\label{def:approx_unknown}
\mathfrak{q}^{\diamond,k}_{\alpha}(f,X)=\sup\{f(d_{\beta}^{k}),\beta \in \cup_{j=1}^{j_{max}(N)}  \Pi^{\diamond,k} (j),\sum_{\gamma \in \{0,\ldots,3^{k}-1\}^d,  f_{\gamma}^{k}\geqslant f(d_{\beta}^{k})} \mathbb{P}(X\in D^{k}_{\gamma}) \geqslant 1- \alpha \} ,
\end{align}
and
\begin{align*}
\Pi^{\diamond,k+1}(j)=\{3 \beta+\{0,1,2\}^d,\beta \in \Pi^{\diamond,k}(j) , f(d^{k}_{\beta}) \in [\mathfrak{q}^{\diamond,k}_{\alpha}- 2L_f (j)\delta^{k},\mathfrak{q}^{\diamond,k}_{\alpha}+ 2 L_f(j) \delta^{k}] \} ,
\end{align*}
when $k+1\leqslant \ell(j,N)$, and 
\begin{align*}
\Pi^{\diamond,k+1}(j)=3\Pi^{\diamond,k}(j)+(1,\ldots,1),
\end{align*}
otherwise.

Notice that similarly to the known Lipischitz constant case, we have
\begin{align} 
\label{def:approx_unknown_bis}
\mathfrak{q}^{\diamond,k}_{\alpha}(f,X)=\inf\{f(d_{\beta}^{k}),\beta \in \cup_{j=1}^{j_{max}(N)}  \Pi^{\diamond,k} (j),\sum_{\gamma \in \{0,\ldots,3^{k}-1\}^d,  f_{\gamma}^{k}\leqslant f(d_{\beta}^{k})} \mathbb{P}(X\in D^{k}_{\gamma}) \geqslant \alpha \},
\end{align}
and $\mathfrak{q}^{\diamond,k}_{\alpha}(f,X)$ is the value returned by our algorithm if it has enough budget to reach level of subdivision $k$ for, at least, one candidate $j \in \{1,\ldots, j_{max}(N)\}$.  

\subsection{Algorithm and Main result}
We now introduce our algorithm designed for the case where the Lipschitz constant  $L_{f}$ of $f$ is unknown. In line with the approach used when the Lipschitz constant is known, the algorithm aims to compute $\mathfrak{q}^{\diamond,k}_{\alpha}(f,X)$ for the highest possible attainable subdivision level $k$.

\begin{algo}
\label{algo:gen_Lip_inconnu}  \quad \\
\noindent \textbf{Fix} $N \in \mathbb{N}^{\ast}$ the maximum number of calls to $f$ \\
$\Pi^{0}=\{(0,\ldots,0)\}$.\\
$k=0$ \\
$J=\{1,\ldots,j_{max}(N)\}$.\\
\textbf{Set} $N_{\mbox{call}}(j)=1$ for all $j \in J$.\\
\textbf{While} $J \neq \emptyset$\\
\vspace{2mm}
\quad Compute $f(d^{k}_{\beta})$, $\beta \in \cup_{j=1}^{j_{max}(N)}\Pi^{\diamond,k} (j)$ (When $d^{k}_{\beta}=d^{k-1}_{\beta^{-[1]}}$, the computation was already done before step $k$).\\
\vspace{2mm}
\quad \textbf{Set} 
$f^{k}_{\beta}=f(d^{k}_{\beta})$, for every $\beta \in \cup_{j=1}^{j_{max}}\Pi^{\diamond,k} (j) $ and $f^{k}_{\beta}=f^{k-1}_{\beta^{-[1]}}$, for every $\beta \notin \cup_{j=1}^{j_{max}}\Pi^{\diamond,k} (j) $ .\\
\vspace{2mm}
\quad\textbf{Set} 
\begin{align*}
q^{\diamond,N}_{\alpha}(f,X)=&\sup\{f(d_{\beta}^{k}) ,\beta \in  \cup_{j=1}^{j_{max}(N)}\Pi^{\diamond,k} (j), \sum_{\gamma \in \{0,\ldots,3^{k}-1\}^d,  f^{k}_{\gamma}\geqslant f(d_{\beta}^{k})} \mathbb{P}(X\in D^{k}_{\gamma}) \geqslant 1- \alpha \}\\
=& \inf\{f(d_{\beta}^{k}) ,\beta \in \cup_{j=1}^{j_{max}(N)}\Pi^{\diamond,k} (j), \sum_{\gamma \in \{0,\ldots,3^{k}-1\}^d,  f_{\gamma}^{k}\leqslant f(d_{\beta}^{k})} \mathbb{P}(X\in D^{k}_{\gamma}) \geqslant \alpha \}.
\end{align*}
\vspace{2mm}
\quad \textbf{For} $j \in J$ \\
\vspace{2mm}
\qquad \textbf{Set} \quad
$\Pi^{\diamond,k+1}(j)=\{3 \beta+\{0,1,2\}^d,\beta \in \Pi^{\diamond,k}(j) , f(d^{k}_{\beta}) \in [q^{\diamond,N}_{\alpha}- 2L_f (j)\delta^{k},q^{\diamond,N}_{\alpha}+ 2 L_f(j) \delta^{k}] \}$.\\
\vspace{2mm}
\qquad  $N_{\mbox{call}}(j)=N_{\mbox{call}}(j)+\frac{3^d-1}{3^d}\mbox{Card}(\Pi^{\diamond,k+1}(j))$. \\
\vspace{2mm}
\qquad \textbf{Set} \textbf{If} $N_{\mbox{call}}(j)>N_{max}(j,N)$. \\
\vspace{2mm}
\qquad \qquad \quad $J=J \setminus \{j\}$.\\
\vspace{2mm}
\quad \textbf{For} $j \in \{1,\ldots,j_{max}(N)\} \setminus J$ \\
\vspace{2mm}
\vspace{2mm}
\qquad \quad $ \Pi^{\diamond,k+1}(j)=3 \Pi^{\diamond,k}(j)+(1,\ldots,1)$.\\
\vspace{2mm}
\quad \textbf{Set}  $k=k+1$. \\
\textbf{End While}\\
\textbf{Return} $q^{\diamond,N}_{\alpha}(f,X)$.

\end{algo}

The following result establishes the convergence of Algorithm \ref{algo:gen_Lip_inconnu} with explicit upper bound for the error. We still exhibit exponential or polynomial regimes regarding that $d=1$ or $d>1$. Nevertheless, as expected, the constants obtained (denoted $\rho$ and $C$ in Theorem \ref{th:conv_gen} and Theorem \ref{th:conv_gen_unknown}) provide a faster convergence to zero for the error bound when the Lipschitz constant is known.

\begin{theorem}
\label{th:conv_gen_unknown}
Let $N \in \mathbb{N}^{\ast}$, and let $q^{\diamond,N}_{\alpha}(f,X)$  be defined as in Algorithm \ref{algo:gen_Lip_inconnu}.  Assume that  (\ref{hyp:lipschitz}) and (\ref{hyp:level_set}) hold with $L_{f} \geqslant 1$. \\

If $d=1$, then

\begin{align*}
\vert q_{\alpha}(f,X) -q^{\diamond,N}_{\alpha}(f,X) \vert 
\leqslant  C \rho^{N} ,
\end{align*}
with $C=18 L_{f} 3^{ \frac{1}{2 M L_{f} }  } $ and $\rho=3^{- \frac{1} { (\frac{\ln( L_{f})}{\ln(3)} +2)^{2}  2  \pi^{2} M L_{f} }  }$.\\

If $d>1$, and $N> \frac{\pi^{2}}{3}(\frac{\ln( L_{f})}{\ln(3)}+2)^{2}$, then

%

\begin{align*}
\vert q_{\alpha}(f,X) -q^{\diamond,N}_{\alpha}(f,X) \vert 
\leqslant C (N-\frac{\pi^{2}}{3}(\frac{\ln( L_{f})}{\ln(3)}+2)^{2})^{- \frac{1}{d-1 } } ,
\end{align*}
with $C= 18 L_{f}  d^{\frac{1}{2}}( 3^{d} M L_{f}d^{\frac{1}{2}}  \frac{\pi^{2}}{2}(\frac{\ln( L_{f})}{\ln(3)}+2)^{2})^{ \frac{1}{d-1 } }$. 

\end{theorem}

\subsection{Proof of convergence of the algorithm}
In this section, we focus on the proof of Theorem \ref{th:conv_gen_unknown}. Our strategy involves selecting the Lipschitz constant candidate which is the closest to the true Lipschitz constant from above. By leveraging the budget allocated to this candidate and the fact that computations for other candidates do not affect convergence, we derive the expected result.\\

Recalling that $L_{f}(j)=3^{j}$, $j \in \mathbb{N}$, we define the (unknown) quantity $j^{\ast} \in \mathbb{N} $, such that $L_{f}(j^{\ast}-1) < L_{f} \leqslant L_{f}(j^{\ast}) $ with convention $L_{f}(-1) =0$. In particular, $ \frac{\ln(L_{f})}{\ln(3)}\leqslant j^{\ast} <  \frac{\ln(3 L_{f})}{\ln(3)} $ since $L_{f} \geqslant 1$. 

The proof of Theorem \ref{th:conv_gen_unknown} is a consequence of Lemma \ref{lem:intervalle_confiance_unknown} thereafter, combined with the fact that $L_{f}(j^{\ast}) \in [L_{f},3 L_{f})$. Indeed, let us introduce $ \ell(j,N) := k(N_{max}(j,N),L_{f}(j))$ when $j \in \{0,\ldots,j_{max}(N)\}$ ($ \ell(j,N):=0$ for $j>j_{max}(N)$). Then, given a budget of $N_{max}(j,N)$ for the Lipschitz constant candidate $L_{f}(j)$, it follows from Lemma \ref{lemme:borne_algo_taille_intervalle_confiance} that the highest level $k$ attained for subdivision of the space satisfies $k \geqslant  \ell(j,N)$. Assuming this highest level is exactly $\ell(j,N)$, the proof of Theorem \ref{th:conv_gen_unknown} is a direct consequence of Lemma \ref{lem:intervalle_confiance_unknown}. When this level is higher, the proof is similar and left to the reader.

\begin{lem}
\label{lem:intervalle_confiance_unknown}
Assume that (\ref{hyp:lipschitz}) holds.  Then, for every $k \in \mathbb{N}$, we have
\begin{align}
\label{eq:intervalle_confiance_unkown}
\vert q_{\alpha}(f,X)- \mathfrak{q}_{\alpha}^{\diamond,k} \vert \leqslant 4L_{f}(j^{\ast}) \delta^{ \min (k, \ell(j^{\ast},N))} .
\end{align}

\end{lem}

\begin{proof}
First, let us consider the case $j^{\ast} > j_{max}(N)$. In this case, $ \mathfrak{q}_{\alpha}^{\diamond,k} \in f([0,1]^{d})$ and the result follows from the Lipschitz property of $f$ since $\ell(j^{\ast},N)=0$. \\ 
Now let us assume that $j^{\ast}  \in \{0,\ldots,j_{max}(N)\}$.
Let us first consider the case $k \leqslant \ell(j^{\ast},N)$. We remark that (using notation from (\ref{def:Pik})), $\Pi_{L_{f}}^{k} \subset \Pi_{L_{f}(j^{\ast})}^{k} =\Pi^{k}(j^{\ast}) \subset \cup_{j=0}^{j_{max}}\Pi^{k}(j)$. In particular, it follows that $ \mathfrak{q}_{\alpha}^{\diamond,k}(f,X) = \mathfrak{q}_{L_{f}(j^{\ast}),\alpha}^{k}(f,X)=\mathfrak{q}_{L_{f},\alpha}^{k}(f,X) = \mathfrak{q}_{\alpha}^{k}(f,X)$, and the result is a direct consequence of Lemma \ref{lem:intervalle_confiance}. \\
We now consider $k >\ell(j^{\ast},N)$. We first notice that if $j > j^{\ast}$, $\ell(j^{\ast},N)\geqslant \ell(j,N)$ and it is impossible that $\mathfrak{q}^{\diamond, \ell(j^{\ast},N)}_{\alpha}(f,X) =f(d_{\gamma}^{\ell(j^{\ast},l)})$ for $\gamma \in \Pi^{\diamond,\ell(j^{\ast},N)}(j) \setminus \Pi^{\diamond,\ell(j^{\ast},N)}(j^{\ast})$. Let us now show that for $k >\ell(j^{\ast},N)$, it is impossible that $\mathfrak{q}^{\diamond, k}_{\alpha}(f,X) =f(d_{\gamma_{\ast}}^{\ell(j^{\ast},l)})$ for $\gamma_{\ast} \in \Pi^{\diamond,\ell(j^{\ast},N)}(j) \setminus \Pi^{\diamond,\ell(j^{\ast},N)}(j^{\ast})$. Indeed,  since $\gamma_{\ast} \notin  \Pi^{\diamond,\ell(j^{\ast},N)}(j^{\ast})$, we have 
\begin{align*}
\vert \mathfrak{q}^{\diamond, \ell(j^{\ast},l)-1}_{\alpha}(f,X) -f(d_{\gamma_{\ast}^{-[1]}}^{\ell(j^{\ast},l)-1}) \vert > 2 L_{f}(j^{\ast}) \delta^{\ell(j^{\ast},l)-1}.
\end{align*}
Moreover, let us notice that $(\mathfrak{q}_{L_{f}(j^{\ast}),\alpha}^{k}-L_{f}(j^{\ast})\delta^{k})_{k \in \mathbb{N}^{\ast}}$ is increasing while $(\mathfrak{q}_{L_{f}(j^{\ast}),\alpha}^{k}+L_{f}(j^{\ast})\delta^{k})_{k \in \mathbb{N}^{\ast}}$ is decreasing. In particular, we have
\begin{align*}
\vert \mathfrak{q}^{\diamond, \ell(j^{\ast},l)}_{\alpha}(f,X) -  \mathfrak{q}^{\diamond, \ell(j^{\ast},N)-1}_{\alpha}(f,X) \vert \leqslant 2 L_{f}(j^{\ast}) \delta^{\ell(j^{\ast},N)} .
\end{align*}
Using the Lipischitz property of $f$ yields
\begin{align*}
\vert f(d_{\gamma_{\ast}^{-[1]}}^{\ell(j^{\ast},l)-1}) -f(d_{\gamma_{\ast}}^{\ell(j^{\ast},l)}) \vert \leqslant  2 L_{f}(j^{\ast}) \delta^{\ell(j^{\ast},N)},
\end{align*}
so, using the triangle inequality, we conclude that
\begin{align*}
\vert \mathfrak{q}^{\diamond, \ell(j^{\ast},l)}_{\alpha}(f,X) -f(d_{\gamma_{\ast}}^{\ell(j^{\ast},l)}) \vert > 2 L_{f} \delta^{\ell(j^{\ast},N)}.
\end{align*}
We now assume that $\mathfrak{q}^{\diamond, \ell(j^{\ast},l)}_{\alpha}(f,X)\leqslant f(d_{\gamma_{\ast}}^{\ell(j^{\ast},l)}) $. It follows that
\begin{align*}
 \sum_{\gamma \in \{0,\ldots,3^{k}-1\}^d,  f^{k}_{\gamma}\geqslant f(d_{\gamma_{\ast}}^{\ell(j^{\ast},l)})} \mathbb{P}(X\in D^{k}_{\gamma})  \leqslant &  \sum_{\gamma \in \{0,\ldots,3^{k}-1\}^d,  f^{\ell(j^{\ast},N)}_{\gamma^{-[k-\ell(j^{\ast},l)]}} + L_{f}(j^{\ast}) \delta^{\ell(j^{\ast},l)} \geqslant f(d_{\gamma_{\ast}}^{\ell(j^{\ast},l)})} \mathbb{P}(X\in D^{k}_{\gamma})   \\
\leqslant &  \sum_{\gamma \in \{0,\ldots,3^{k}-1\}^d,  f^{\ell(j^{\ast},N)}_{\gamma^{-[k-\ell(j^{\ast},l)]}xt}\geqslant f(d_{\gamma_{\ast}}^{\ell(j^{\ast},l)}) - L_{f}(j^{\ast}) \delta^{\ell(j^{\ast},l)} } \mathbb{P}(X\in D^{k}_{\gamma})   \\
\leqslant &  \sum_{\gamma \in \{0,\ldots,3^{k}-1\}^d,  f^{\ell(j^{\ast},N)}_{\gamma^{-[k-\ell(j^{\ast},l)]}} > \mathfrak{q}^{\diamond, \ell(j^{\ast},l)}_{\alpha}(f,X) } \mathbb{P}(X\in D^{k}_{\gamma})  \\
= &  \sum_{\gamma \in \{0,\ldots,3^{\ell(j^{\ast},N)}-1\}^d,  f^{\ell(j^{\ast},N)}_{\gamma} > \mathfrak{q}^{\diamond, \ell(j^{\ast},l)}_{\alpha}(f,X) } \mathbb{P}(X\in D^{k}_{\gamma})  \\
< & 1- \alpha,
\end{align*}
where the last inequality is a direct consequence of definition (\ref{def:approx_unknown}), and we conclude that $\mathfrak{q}^{\diamond, k}_{\alpha}(f,X) \neq f(d_{\gamma_{\ast}}^{\ell(j^{\ast},l)})$. When $\mathfrak{q}^{\diamond, \ell(j^{\ast},l)}_{\alpha}(f,X)\geqslant f(d_{\gamma_{\ast}}^{\ell(j^{\ast},l)}) $, we use the same approach, but apply definition (\ref{def:approx_unknown_bis}) instead of (\ref{def:approx_unknown}).
Now, we remark that for every $j \in \{0,\ldots,j^{\ast}\}$, we have $\Pi^{\diamond,\ell(j^{\ast})}(j) \subset  \Pi^{\diamond,\ell(j^{\ast})}(j^{\ast})$. It follows from the previous observation that for every $k \geqslant \ell(j^{\ast},N)$, 
\begin{align*}
\mathfrak{q}^{\diamond, k}_{\alpha}(f,X) \in f(\cup_{\beta \in \Pi^{\ell(j^{\ast},N)}} D^{\ell(j^{\ast},N)}_{\beta}) ,
\end{align*}
and
\begin{align*}
\vert \mathfrak{q}^{\diamond,k}_{\alpha}(f,X)  -  \mathfrak{q}^{\diamond, \ell(j^{\ast},N)}_{\alpha}(f,X)  \vert  \leqslant & 
\sup_{\beta \in \Pi^{\ell(j^{\ast},N)}}\vert f(d^{\ell(j^{\ast},N)}_{\beta}) - \mathfrak{q}^{\diamond,\ell(j^{\ast},N)}_{\alpha}(f,X)  \vert \\
& + \sup_{\beta \in \Pi^{\ell(j^{\ast},N)}} \sup_{x \in  D^{\ell(j^{\ast},N)}_{\beta}}  \vert f(d^{\ell(j^{\ast},N)}_{\beta}) - f(x) \vert   \\
\leqslant & 3 L_{f}(j^{\ast}) \delta^{\ell(j^{\ast},N)} .
\end{align*}
Hence, (\ref{eq:intervalle_confiance_unkown}) is obtained by applying the triangle inequality and Lemma \ref{lem:intervalle_confiance} with $k=\ell(j^{\ast},N)$ (remember that $ \mathfrak{q}^{\diamond, \ell(j^{\ast},N)}_{\alpha}(f,X)  = \mathfrak{q}^{\ell(j^{\ast},N)}_{L_{f}(j^{\ast}),\alpha}(f,X)  $).
\end{proof}

\section{Optimality}
\label{sec:opt}
In this section, we aim to demonstrate the optimality of our algorithms. Specifically, we demonstrate that, within our framework where (\ref{hyp:lipschitz}) and (\ref{hyp:level_set}) are assumed, it is impossible to construct lower and upper deterministic bounds for $q_{\alpha}(f,X)$ using $N$ evaluation of $f$, such that the error (essentially the difference between upper and lower bounds), converges faster than exponential or polynomial rate (given the value of $d$) $w.r.t.$ $N$.

Our approach consists in considering a generic algorithm which evaluates $N \in \mathbb{N}^{\ast}$ times the function $f:\mathbb{R}^{d} \to \mathbb{R}$ at some points $(x^{1},\ldots,x^{N}) \in ([0,1]^{d})^{N}$ and which returns a measurable function of those evaluations to compute $q_{\alpha}(f,X)$, $\alpha \in (0,1)$ with $f$, $X$ and $\alpha$ satisfying (\ref{hyp:lipschitz}) and (\ref{hyp:level_set}). 
We then propose a function $\bar{f}$, a random variable $X$ and $\alpha \in (0,1)$ satisfying (\ref{hyp:lipschitz}) and (\ref{hyp:level_set}), and for any choice of $(x_{1},\ldots,x_{N})$, we build $\tilde{f}$ such that $\tilde{f}$, $X$ and $\alpha$ satisfies (\ref{hyp:lipschitz}) and (\ref{hyp:level_set}), but also $\bar{f}(x_{i})=\tilde{f}(x_{i})$ for every $i \in \{1,\ldots,N\}$, and $\vert q_{\alpha}(f,X)-q_{\alpha}(\tilde{f},X) \vert \geqslant C \rho^{N}$ ($\rho \in (0,1)$ and $d=1$) or $C N^{-\frac{1}{d-1}}$ (when d>1), $C>0$. This property is combined with the following observation: For any measurable $g:\mathbb{R}^{N} \to \mathbb{R}$, 

\begin{align*}
\vert g(\bar{f}(x^{1}),\ldots,\bar{f}(x^{N})) - q_{\alpha}(f,X) \vert = &  \vert g(\tilde{f}(x^{1}),\ldots,\tilde{f}(x^{N})) - q_{\alpha}(\bar{f},X) \vert \\
\geqslant &    \vert q_{\alpha}(\tilde{f},X)- q_{\alpha}(\bar{f},X) \vert  -  \vert g(\tilde{f}(x^{1}),\ldots,\tilde{f}(x^{N})) - q_{\alpha}(\tilde{f},X) \vert  .
\end{align*}
It then follows that
\begin{align*}
\max(\vert g(\bar{f}(x^{1}),\ldots,\bar{f}(x^{N})) - q_{\alpha}(\bar{f},X) \vert, \vert g(\tilde{f}(x^{1}),\ldots,\tilde{f}(x^{N})) - q_{\alpha}(\tilde{f},X) \vert  ) 
\geqslant &   \frac{1}{2} \vert q_{\alpha}(\tilde{f},X)- q_{\alpha}(\bar{f},X) \vert .
\end{align*}
In other words, for any choice of measurable function $g$ (referred to as an algorithm), based on $N$ evaluations, we can construct two triplets $(\bar{f},X,\alpha)$ and $(\tilde{f},X,\alpha)$ satisfying both (\ref{hyp:lipschitz}) and (\ref{hyp:level_set}) and such that, the error bounds studied in Theorems \ref{th:conv_gen} and \ref{th:conv_gen_unknown} (with $q^{N}_{\alpha}(f,X)$ or $q^{\diamond,N}_{\alpha}(f,X)$ replaced by $ g(f(x^{1}),\ldots,f(x^{N})$) for $f \in \{\bar{f},\tilde{f}\}$), for at least one triplet, are lower bounded by $C \rho^{N}$ ($\rho \in (0,1)$) or $C N^{-\frac{1}{d-1}}$, depending on the value of $d$. Thus, the exponential or polynomial convergence rates established in these theorems cannot be improved within our framework.

\subsection{Case $d>1$.}
In this section, we are interested in the optimality of the polynomial convergence obtained in Theorems \ref{th:conv_gen} and \ref{th:conv_gen_unknown} for the case $d>1$.
\begin{prop}
\label{prop:optimality_d>1}
Let $d >1$. There exists $C>0$ such that for every $N \in \mathbb{N}^{\ast}$, every $g:\mathbb{R}^{N} \to \mathbb{R}$ measurable and every $(x^{1},\ldots,x^{N}) \in ([0,1]^{d})^{N}$, there exists $f,X,\alpha$ such that (\ref{hyp:lipschitz}) and (\ref{hyp:level_set}) hold and that
\begin{align*}
\vert g(f(x^{1}),\ldots,f(x^{N})) - q_{\alpha}(f,X) \vert \geqslant C N^{\frac{1}{d-1}}.
\end{align*}

\end{prop}

\begin{remark}
This result shows that the rate of convergence obtained in Theorems \ref{th:conv_gen} and \ref{th:conv_gen_unknown} in the case $d>1$ is optimal. In other words, any algorithm other than ours can only improve (lower in this case) the constant $C$ obtained in our deterministic upper bound of the error, but not the polynomial rate $N^{\frac{1}{d-1}}$. We notice that it remains compatible with the rate of convergence $N^{\frac{1}{2}}$ of Monte Carlo, because the latter is a weak rate of convergence.
\end{remark}
\begin{proof}
As explained at the beginning of this section, to prove this result, it is sufficient to build $\bar{f},\tilde{f}:[0,1]^{d} \to \mathbb{R}$, $X$ and $\alpha$ such that $(\bar{f},X,\alpha)$ and $(\tilde{f},X,\alpha)$ both satisfy (\ref{hyp:lipschitz}) and (\ref{hyp:level_set}) and that $f(x_{i})=\tilde{f}(x_{i})$ for every $i \in \{1,\ldots,N\}$ and $\vert q_{\alpha}(f,X)-q_{\alpha}(\tilde{f},X) \vert \geqslant C N^{-\frac{1}{d-1}}$, $C>0$.

Let $X \sim \mathcal{U}_{[0,1]^{d}}$ be a uniform random variable on $[0,1]^{d}$, let $\alpha=\frac{1}{2}$ and let us introduce $\bar{f}:[0,1]^{d} \to \mathbb{R}$ such that $\bar{f}(x)=x_{1}$ for every $x=(x_{1},\ldots ,x_{d})\in [0,1]^{d}$. In this case $q_{\alpha}(\bar{f},X)=\alpha=\frac{1}{2}$. 

To begin, it is straightforward to verify that $\bar{f}$ satisfies assumption (\ref{hyp:lipschitz}) with $L_{\bar{f}}=1$. Let us verify assumption (\ref{hyp:level_set}): Let $\delta >0$. Since $q_{\alpha}(\bar{f},X)=\frac{1}{2}$, we have
\begin{align*}
\lambda_{\mbox{Leb}}(x \in [0,1]^{d},\bar{f}(x) \in [q_{\alpha}(\bar{f},X)-\delta,q_{\alpha}(\bar{f},X)+\delta])  = & \lambda_{\mbox{Leb}}(x \in [0,1]^{d},x_{1} \in [\frac{1}{2}-\delta,\frac{1}{2}+\delta])   =  2 \delta,
\end{align*}
and $(\ref{hyp:level_set})$ holds for $\bar{f},X$ and $\alpha=\frac{1}{2}$ with $M=2$. 

We now propose a construction for $\tilde{f}$. We introduce
\begin{align*}
\mathcal{C}:=\{x \in [0,1]^{d},\bar{f}(x)=\frac{1}{2} \}=\{x \in [0,1]^{d},x_{1}=\frac{1}{2} \}.
\end{align*}
We focus on the case $N= \frac{3^{j(d-1)}}{3}$, the proof being similar otherwise. Let us define
\begin{align*}
\hat{D}^{j}:=\{D^{j}_{\beta}, D^{j}_{\beta} \cap \mathcal{C} \neq \emptyset ,\beta \in \{0,\ldots,3^{j}-1\}^d \} ,
\end{align*}
and
\begin{align*}
\tilde{D}^{j}:=\{D^{j}_{\beta}, D^{j}_{\beta} \cap \mathcal{C} \neq \emptyset,  \bigcap_{i=1}^{N} x^{i} \notin D^{j}_{\beta},\beta \in \{0,\ldots,3^{j}-1\}^d \}.
\end{align*}
In particular, $j$ is defined in a way such that $\mbox{Card}(\hat{D}^{j})=3^{j(d-1)}=3N$ and then $\mbox{Card}(\tilde{D}^{j}) \geqslant 2N$. We are now in a position to introduce $\tilde{f}$. Let $\mathbf{L} >1$. For every $x \in [0,1]^{d}$, let us define
\begin{align*}
\tilde{f}(x)=\bar{f}(x)+ \mathbf{L}  \sum_{ D^{j}_{\beta} \in \tilde{D}^{j}} \inf_{y \notin D^{j}_{\beta} } \vert x -y  \vert_{\infty} , 
\end{align*}
with $ \vert x -y  \vert_{\infty}:=\sup_{i \in \{1,\ldots,d\}} \vert x_{i}- y_{i} \vert$. 

We remark that the piecewise affine function $h:x \mapsto  \sum_{ D^{j}_{\beta} \in \tilde{D}^{j}}  \inf_{y \notin D^{j}_{\beta} } \vert x -y  \vert_{\infty}$ and $\bar{f}$ are both 1-Lipschitz so that (\ref{hyp:lipschitz}) holds for $\tilde{f}$ with $L_{\tilde{f}}=\mathbf{L} +1$.  Let us now show that (\ref{hyp:level_set}) holds for $\tilde{f},X$ and $\alpha=\frac{1}{2}$ with $M<+\infty$.  First, decomposing our computation on $\tilde{D}^{j}$ and its complementary space, we obtain
\begin{align*}
\lambda_{\mbox{Leb}}(\tilde{f}(x)\in &  [q_{\frac{1}{2}}(\tilde{f},X)- \delta,q_{\frac{1}{2}}(\tilde{f},X)+\delta]) =  \sum_{ D^{j}_{\beta} \notin \tilde{D}^{j}}  \lambda_{\mbox{Leb}}(x \in D^{j}_{\beta},\tilde{f}(x)\in [q_{\frac{1}{2}}(\tilde{f},X)-\delta,q_{\frac{1}{2}}(\tilde{f},X)+\delta])  \\
& + \sum_{ D^{j}_{\beta} \in \tilde{D}^{j}}  \lambda_{\mbox{Leb}}(x \in D^{j}_{\beta},\tilde{f}(x)\in [q_{\frac{1}{2}}(\tilde{f},X)-\delta,q_{\frac{1}{2}}(\tilde{f},X)+\delta])   \\
= & \lambda_{\mbox{Leb}}(x \in \cup_{D^{j}_{\beta} \notin \tilde{D}^{j}}D^{j}_{\beta}, x_{1}\in [q_{\frac{1}{2}}(\tilde{f},X)-\delta,q_{\frac{1}{2}}(\tilde{f},X)+\delta])  \\
& + \sum_{ D^{j}_{\beta} \in \tilde{D}^{j}}  \lambda_{\mbox{Leb}}(,\tilde{f}(x)\in [q_{\frac{1}{2}}(\tilde{f},X)-\delta,q_{\frac{1}{2}}(\tilde{f},X)+\delta])  \\
\leqslant & 2 \delta + \sum_{ D^{j}_{\beta} \in \tilde{D}^{j}}  \lambda_{\mbox{Leb}}(x \in D^{j}_{\beta}, x_{1}+\mathbf{L} \inf_{y \notin D^{j}_{\beta} } \vert x -y  \vert_{\infty}  \in [q_{\frac{1}{2}}(\tilde{f},X)-\delta,q_{\frac{1}{2}}(\tilde{f},X)+\delta]) .
\end{align*}

Now, for $D^{j}_{\beta} \in  \tilde{D}^{j}$, we remark that $x \mapsto x_{1}+\mathbf{L} \inf_{y \notin D^{j}_{\beta} } \vert x -y  \vert_{\infty}$ defined on  $\tilde{D}^{j}_{\beta}$ is a rescaled version of the function $f_{\ast}:x \mapsto x_{1}+\mathbf{L} \inf_{y \notin [0,1]^{d} } \vert x -y  \vert_{\infty}$, $x \in[0,1]^{d}$. More particularly, for every $x \in D^{j}_{\beta}$, $\tilde{f}(x)=\frac{1}{2}-\frac{3^{-j}}{2}+3^{-j}f_{\ast}( 3^{j}  (x-d^{j}_{\beta}+\frac{3^{-j}}{2}(1,\ldots,1) ))$. We also observe that $f_{\ast}$ is $a.e$ differentiable, has a piecewise constant gradient and satisfies $ \vert \nabla f_{\ast}(x) \vert_{\mathbb{R}^{d}} \geqslant \frac{\mathbf{L}-1}{\sqrt{d}}$ for $a.e.$ $x \in [0,1]^{d}$.  In addition, it follows from the coarea formula (see Theorem 3.11 in \cite{Evans_Gariepy_2015}) that for every $a\geqslant 0$ and $\delta>0$,
\begin{align*}
\lambda_{\mbox{Leb}}(x \in [0,1]^{d}, f_{\ast}(x) \in [a-\delta,a+\delta])= \int_{a-\delta}^{a+\delta}  \int_{f_{\ast}^{-1}(\{y \})\cap [0,1]^{d}} \frac{1}{\vert \nabla f_{\ast} (z)\vert } \mathcal{H}^{d-1}( \mbox{d}z)  \mbox{d}y,
\end{align*}
where $\mathcal{H}^{d-1}$ stands for the $(d-1)$-dimensional Hausdorff measure on $\mathbb{R}^{d}$. Moreover, $f_{\ast}$ is equal to $q_{\frac{1}{2}}(\tilde{f},X)$ on a polyhedral shaped set ($i.e.$, $f_{\ast}^{-1}(\{q_{\frac{1}{2}}(\tilde{f},X) \})$ with finite $(d-1)$-dimensional Hausdorff measure that we denote $H$. Hence, taking $a=q_{\frac{1}{2}}(\tilde{f},X)$ and $\delta $ small enough in the coarea formula yields
\begin{align*}
\lambda_{\mbox{Leb}}(x \in [0,1]^{d}, f_{\ast}(x) \in [q_{\frac{1}{2}}(\tilde{f},X)-\delta,q_{\frac{1}{2}}(\tilde{f},X)+\delta]) \leqslant 2 \delta (H+1) \frac{\sqrt{d}}{\mathbf{L}-1}.
\end{align*}
Therefore, there exists $M_{\ast}$ such that for every $\delta >0$, we have
\begin{align*}
\lambda_{\mbox{Leb}}(x \in [0,1]^{d}, f_{\ast}(x) \in [q_{\frac{1}{2}}(\tilde{f},X)-\delta,q_{\frac{1}{2}}(\tilde{f},X)+\delta]) \leqslant M_{\ast} \delta.
\end{align*}
Recalling that $\tilde{f}$ is a rescaled version of $f_{\ast}$ and using the change of variables formula, it follows that for $D^{j}_{\beta} \in  \tilde{D}^{j}$,  
\begin{align*}
\lambda_{\mbox{Leb}}(x \in D^{j}_{\beta} , \tilde{f}(x)& \in [q_{\frac{1}{2}}(\tilde{f},X)-\delta,q_{\frac{1}{2}}(\tilde{f},X)+\delta]) \\
=&  \lambda_{\mbox{Leb}}(x \in D^{j}_{\beta} ,3^{-j}f_{\ast}( 3^{j}  (x-d^{j}_{\beta}+\frac{3^{-j}}{2}(1,\ldots,1) )) \in [3^{-j}q_{\frac{1}{2}}(\tilde{f},X)-\delta,3^{-j}q_{\frac{1}{2}}(\tilde{f},X)+\delta])   \\
=&  3^{-jd}\lambda_{\mbox{Leb}}(x \in [0,1]^{d},3^{-j}f_{\ast}(x) \in [3^{-j}q_{\frac{1}{2}}(\tilde{f},X)-\delta,3^{-j}q_{\frac{1}{2}}(\tilde{f},X)+\delta])    \\
=& M_{\ast} \delta 3^{-j(d-1)}.
\end{align*}
Summing over all $D^{j}_{\beta} \in  \tilde{D}^{j}$, we obtain
\begin{align*}
 \sum_{ D^{j}_{\beta} \in \tilde{D}^{j}}  \lambda_{\mbox{Leb}}(x_{1}+\mathbf{L} \inf_{y \notin D^{j}_{\beta} } \vert x -y  \vert_{\infty}  \in [q_{\frac{1}{2}}(\tilde{f},X)-\delta,q_{\frac{1}{2}}(\tilde{f},X)+\delta],x \in D^{j}_{\beta}) \leqslant M_{\ast} \delta ,
\end{align*}
and (\ref{hyp:level_set}) holds for $\tilde{f},X$ and $\alpha=\frac{1}{2}$ with $M=M_{\ast}+2 <+\infty$. 

Moreover, for every $i \in \{1,\ldots,N\}$, we have $x^{i} \notin D^{j}_{\beta}$ when $D^{j}_{\beta} \in \tilde{D}^{j}$, so that $\bar{f}(x^{i})=\tilde{f}(x^{i})$.  In order to conclude the proof, we are left to show that $\vert q_{\alpha}(f,X)-q_{\alpha}(\tilde{f},X) \vert \geqslant C N^{-\frac{1}{d-1}}$ for $C>0$. We thus aim to prove that there exists $C>0$ and $\mathbf{L}$ which does not depend on $N$ and such that
\begin{align*}
\mathbb{P}(\tilde{f}(X) \leqslant \frac{1}{2} + C N^{-\frac{1}{d-1}}) <\frac{1}{2}.
\end{align*}
This implies that $q_{\frac{1}{2}}(\tilde{f},X) > \frac{1}{2} + C N^{-\frac{1}{d-1}}=q_{\frac{1}{2}}(\bar{f},X) + C N^{-\frac{1}{d-1}} $, which is the expected conclusion.

From now, let us consider the fixed value $C=\frac{1}{12 (3^{d-1}-1) 3^{\frac{1}{d-1}}}$. We begin by noticing that 
\begin{align*}
\mathbb{P}(\tilde{f}(X) \leqslant \frac{1}{2} + C N^{-\frac{1}{d-1}}) = & \sum_{\beta \in \{0,\ldots , 3^{j}-1\}^{d} } \mathbb{P}(\tilde{f}(X) \leqslant \frac{1}{2} + C N^{-\frac{1}{d-1}} \vert X \in D^{j}_{\beta})  \mathbb{P}(X \in D^{j}_{\beta})  \\
\leqslant & \sum_{D^{j}_{\beta} \notin \tilde{D}^{j}} \mathbb{P}(X_{1} \leqslant \frac{1}{2} + C N^{-\frac{1}{d-1}} \vert X \in D^{j}_{\beta})  \mathbb{P}(X \in D^{j}_{\beta}) \\
& +  \sum_{D^{j}_{\beta} \in \tilde{D}^{j}} \mathbb{P}(X_{1}  + \mathbf{L} \inf_{y \notin D^{j}_{\beta} } \vert X -y  \vert_{\infty}  \leqslant \frac{1}{2} + C N^{-\frac{1}{d-1}}  \vert X \in D^{j}_{\beta})  \mathbb{P}(X \in D^{j}_{\beta})  .
\end{align*}

We now study these two terms in the $r.h.s.$ above. On the one hand, we observe that, when $D^{j}_{\beta} \notin \hat{D}^{j}$, we have $\vert d^{j}_{\beta_{1}}-\frac{1}{2} \vert \geqslant 3^{-j}$, and since $C N^{-\frac{1}{d-1}} \leqslant  \frac{3^{-j}}{2}$, we obtain

\begin{align*}
\mathbb{P}(X_{1} \leqslant \frac{1}{2} + C N^{-\frac{1}{d-1}} \vert X \in D^{j}_{\beta})  =  \mathds{1}_{d^{j}_{\beta_{1}} < \frac{1}{2} } , 
\end{align*}

and 
\begin{align*}
 \sum_{D^{j}_{\beta} \notin \hat{D}^{j}} \mathbb{P}(X_{1} \leqslant \frac{1}{2} + C N^{-\frac{1}{d-1}} \vert X \in D^{j}_{\beta})  \mathbb{P}(X \in D^{j}_{\beta}) =  \mathbb{P}(X_{1} \leqslant \frac{1}{2}  -\frac{3^{-j}}{2})=\frac{1-3^{-j}}{2} .
\end{align*}
It follows that


\begin{align*}
 \sum_{D^{j}_{\beta} \notin \tilde{D}^{j}} \mathbb{P}(X_{1} \leqslant  \frac{1}{2} + C N^{-\frac{1}{d-1}} \vert X \in D^{j}_{\beta})  \mathbb{P}(X \in D^{j}_{\beta}) \leqslant & \frac{1-3^{-j}}{2}+ N3^{-j d}3^{j}(C N^{-\frac{1}{d-1}} + \frac{3^{-j}}{2})  \\
=&  \frac{1-3^{-j}}{2} +\frac{3^{-j}}{3}=  \frac{1}{2} -\frac{3^{-j}}{6}  .
\end{align*}


On the other hand, we are going to show that
\begin{align*}
 \sum_{D^{j}_{\beta} \in \tilde{D}^{j}} \mathbb{P}(X_{1}  + \mathbf{L} \inf_{y \notin D^{j}_{\beta} } \vert X -y  \vert_{\infty}  \leqslant \frac{1}{2} + C N^{-\frac{1}{d-1}}  \vert X \in D^{j}_{\beta})  \mathbb{P}(X \in D^{j}_{\beta}) \leqslant \frac{3^{-j}}{6},
\end{align*}
and the proof will be completed. For $u \in (0, \frac{3^{-j}}{2})$,  $\beta \in \{0,\ldots,3^{j}-1\}^{d}$, let us denote $D^{j}_{\beta,u}=[d^{j}_{\beta_{1}}-u,d^{j}_{\beta_{1}}+u] \times \ldots \times [d^{j}_{\beta_{d}}-u,d^{j}_{\beta_{d}}+u]  $. We remark that, when $D^{j}_{\beta} \in \tilde{D}^{j}$, then for every $x \in D^{j}_{\beta,u}$, we have $\tilde{f}(x) \leqslant x_{1} + \mathbf{L}(\frac{3^{-j}}{2}- u )$ and $d^{j}_{\beta_{1}}=\frac{1}{2}$. Using the independence of the components of $X$, it follows that for $D^{j}_{\beta} \in \tilde{D}^{j}$,
\begin{align*}
 \mathbb{P}  (X_{1}  +& \mathbf{L} \inf_{y \notin D^{j}_{\beta} } \vert X -y  \vert_{\infty}  \leqslant \frac{1}{2} + C N^{-\frac{1}{d-1}}  \vert X \in D^{j}_{\beta})    \\
 \leqslant &  \mathbb{P}(X_{1}   \leqslant \frac{1}{2} + C N^{-\frac{1}{d-1}},X  \notin D^{j}_{\beta,u} \vert X \in D^{j}_{\beta})  \\
 &+ \mathbb{P}(X_{1}   \leqslant \frac{1}{2} + C N^{-\frac{1}{d-1}}-\mathbf{L}(\frac{3^{-j}}{2}- u ),X  \in D^{j}_{\beta,u} \vert X \in D^{j}_{\beta}) \\
=&  \mathbb{P}(X_{1}   \leqslant \frac{1}{2} + C N^{-\frac{1}{d-1}},X_{1} \in [\frac{1}{2}-u,\frac{1}{2}+u] \vert X \in D^{j}_{\beta}) \mathbb{P}(\bigcup_{i=2}^{d} X_{i} \notin [d^{j}_{\beta_{i}}-u,d^{j}_{\beta_{i}}+u ] \vert X \in D^{j}_{\beta})   \\
& + \mathbb{P}(X_{1}   \leqslant \frac{1}{2} + C N^{-\frac{1}{d-1}},X_{1} \notin [\frac{1}{2}-u,\frac{1}{2}+u]  \vert X \in D^{j}_{\beta})  \\
 &+ \mathbb{P}(X_{1}   \leqslant \frac{1}{2} + C N^{-\frac{1}{d-1}}-\mathbf{L}(\frac{3^{-j}}{2}- u ),X  \in D^{j}_{\beta,u} \vert X \in D^{j}_{\beta})  \\ 
 \leqslant &  \frac{1}{2}-\frac{1}{2} (2u3^{j})^{d}+3^{j} \min(u, C N^{-\frac{1}{d-1}}) (1-(2u 3^{j})^{d-1}) \\
 &+3^{j}(C N^{-\frac{1}{d-1}}-\mathbf{L}(\frac{3^{-j}}{2}- u )+u)_{+}+3^{j}(C N^{-\frac{1}{d-1}}-u)_{+}.
\end{align*}

Since $C N^{-\frac{1}{d-1}} \leqslant  \frac{3^{-j}}{12(3^{d-1}-1)}$, we choose $u=\frac{\mathbf{L}}{2(\mathbf{L}+1)}3^{-j}$ and $\mathbf{L}>\frac{5^{\frac{1}{d}}}{6^{\frac{1}{d}}-5^{\frac{1}{d}}}$, so we obtain, when $D^{j}_{\beta} \in \tilde{D}^{j}$,
\begin{align*}
 \mathbb{P} & (X_{1}  +\mathbf{L} \inf_{y \notin  D^{j}_{\beta} } \vert X -y  \vert_{\infty}  \leqslant \frac{1}{2} + C N^{-\frac{1}{d-1}}  \vert X \in D^{j}_{\beta}) < \frac{1}{6}.
\end{align*}
Gathering all the terms together yields

\begin{align*}
\sum_{D^{j}_{\beta} \in \tilde{D}^{j}} \mathbb{P}(X_{1}  + \mathbf{L} \inf_{y \notin  D^{j}_{\beta} } \vert X -y  \vert_{\infty}  \leqslant \frac{1}{2} + C N^{-\frac{1}{d-1}}  \vert X \in D^{j}_{\beta})  \mathbb{P}(X \in D^{j}_{\beta})   <  3N 3^{-jd} \frac{1}{6}= \frac{3^{-j}}{6},
\end{align*}

and
\begin{align*}
\mathbb{P}(\tilde{f}(X) \leqslant \frac{1}{2} + C N^{-\frac{1}{d-1}}) < \frac{1}{2}.
\end{align*}

This implies that $q_{\frac{1}{2}}(\tilde{f},X) > \frac{1}{2} + C N^{-\frac{1}{d-1}}$ (with $C=\frac{1}{12 (3^{d-1}-1) 3^{\frac{1}{d-1}}}$) and the proof is completed.
\end{proof}

\subsection {Case $d=1$}
In this Section, we are interested in the optimality of the exponential convergence obtained in Theorems \ref{th:conv_gen} and \ref{th:conv_gen_unknown} in the case $d=1$.

\begin{prop}
\label{prop:optimality_d=1}
Let $d=1$. There exists $C>0$ and $\rho \in (0,1)$ such that for every $N \in \mathbb{N}^{\ast}$, every $g:\mathbb{R}^{N} \to \mathbb{R}$ measurable and every $(x^{1},\ldots,x^{N}) \in [0,1]^{N}$, there exists $f,X,\alpha$ such that (\ref{hyp:lipschitz}) and (\ref{hyp:level_set}) hold and that
\begin{align*}
\vert g(f(x^{1}),\ldots,f(x^{N})) - q_{\alpha}(f,X) \vert \geqslant C \rho^{N}.
\end{align*}

\end{prop}

\begin{remark}
This result shows that the exponential convergence obtained for the upper bound of the error studied in Theorems \ref{th:conv_gen} and \ref{th:conv_gen_unknown} in the case $d=1$ is optimal in the sense that Proposition \ref{prop:optimality_d=1} provides a minorant for this error bound which also converges with exponential rate.
\end{remark}

\begin{proof}
To prove this result, we adopt a similar strategy to the proof of Proposition \ref{prop:optimality_d>1}  in the case $d>1$. We aim to build $\bar{f},\tilde{f}:[0,1]^{d} \to \mathbb{R}$, $X$ and $\alpha$ such that $(\bar{f},X,\alpha)$ and $(\tilde{f},X,\alpha)$ both satisfy (\ref{hyp:lipschitz}) and (\ref{hyp:level_set}) and that $f(x_{i})=\tilde{f}(x_{i})$ for every $i \in \{1,\ldots,N\}$ and $\vert q_{\alpha}(f,X)-q_{\alpha}(\tilde{f},X) \vert \geqslant C \rho^{N}$, $C>0$, $\rho \in (0,1)$. 

We consider $X \sim \mathcal{U}_{[0,1]}$, a uniform random variable on $[0,1]$ and introduce $\bar{f}:[0,1] \to \mathbb{R}$ such that $\bar{f}(x)=x$. We choose $\alpha=\frac{1}{2}$ so that $q_{\alpha}(\bar{f},X)=\frac{1}{2}$. We can easily verify that assumptions (\ref{hyp:lipschitz}) with $L_{\bar{f}}=1$ and $(\ref{hyp:level_set})$ with $M=2$ are satisfied by $\bar{f},X$ and $\alpha=\frac{1}{2}$. 


We now propose a construction for $\tilde{f}$. For $\boldsymbol{\rho} \in (0,1)$, let us define $I_{N+1}:=[\frac{1}{2}(1-\boldsymbol{\rho}^{N}),\frac{1}{2}(1+\boldsymbol{\rho}^{N})]$, and for $j \in \{1,\ldots,N\}$, $I_{j,-} := [\frac{1}{2}(1-\boldsymbol{\rho}^{j-1}),\frac{1}{2}(1-\boldsymbol{\rho}^{j})] $, $  I_{j,+}:= [\frac{1}{2}(1+\boldsymbol{\rho}^{j}),\frac{1}{2}(1+\boldsymbol{\rho}^{j-1})]$ and $I_{j}:=I_{j,-} \cup I_{j,+}$. We also define

\begin{align*}
\hat{D}^{N}:=\{I_{j} , j \in \{1,\ldots,N+1\} \} ,
\end{align*}
and
\begin{align*}
\tilde{D}^{N}:=\{I \in \hat{D}^{N} ,  \bigcap_{i=1}^{N} x^{i} \notin I  \}.
\end{align*}
We observe that $\mbox{Card}(\hat{D}^{N})=N+1$ and then $\mbox{Card}(\tilde{D}^{j}) \geqslant 1$. We are now in a position to introduce $\tilde{f}$. Let us consider an arbitrary $I_{0} \in \tilde{D}^{N}$ and let $\mathbf{L} >1$ and define, for every $x \in [0,1]^{d}$,
\begin{align*}
\tilde{f}(x)=\bar{f}(x)+ \mathbf{L} \inf_{y  \notin I_{0} } \vert x -y  \vert  .
\end{align*}
We now show that (\ref{hyp:lipschitz}) and (\ref{hyp:level_set}) are satisfied by $\tilde{f}$,$X$ and $\alpha=\frac{1}{2}$. We first remark that, the functions $\bar{f}$ and $x \mapsto  \inf_{y \notin I_{0} } \vert x -y  \vert$ are 1-Lipischitz so that (\ref{hyp:lipschitz}) holds with $L_{\tilde{f}}=\mathbf{L}+1$.  Moreover, $\tilde{f}$ is $a.e.$ differentiable with piecewise constant derivatives and since $\mathbf{L}>1$, we have  $\vert \frac{\mbox{d}}{\mbox{d}x}\tilde{f}(x) \vert \geqslant \frac{1}{\mathbf{L}-1}>0 $ for $a.e.$ $x \in[0,1]$. Since $\tilde{f}^{-1}(q_{\frac{1}{2}}(\tilde{f},X)\})$ is finite, it follows from the coarea formula that (\ref{hyp:level_set}) holds for $\tilde{f},X$ and $\alpha=\frac{1}{2}$ with $M<+\infty$. 
We also observe that we have $\bar{f}(x^{i})=\tilde{f}(x^{i})$ for every $i \in \{1,\ldots,N \}$.  In order to conclude the proof, we show that, we can find $C>0$, $\rho \in(0, 1)$ and $\mathbf{L} \in(\frac{1+\rho}{1-\rho},+\infty)$ (which do not depend on $N$) such that $\vert q_{\alpha}(\bar{f},X)-q_{\alpha}(\tilde{f},X) \vert \geqslant C \rho^{N}$. To this end, we prove that,
\begin{align*}
\mathbb{P}(\tilde{f}(X) \leqslant \frac{1}{2} + C \rho^{N})<\frac{1}{2}.
\end{align*}

 First, we observe that, since $\tilde{f}$ is $\mathbf{L}+1$-Lipschitz, we have, for $I_{0} \in \{ I_{1}, \ldots I_{N} \}$,
\begin{align*}
\mathbb{P}(\tilde{f}(X) \leqslant \frac{1}{2} + C \rho^{N}) 
=&  \mathbb{P}(X \leqslant \frac{1}{2} + C \rho^{N} , X  \notin I_{0}) \\
& + \mathbb{P}(X  + \mathbf{L} \inf_{y \notin I_{0}} \vert X -y  \vert  \leqslant \frac{1}{2} + C \rho^{N} , X \in I_{0})  \\
\leqslant &  \mathbb{P}(X \leqslant \frac{1}{2} + C \rho^{N} , X \notin  I_{0}) \\
& + \mathbb{P}(\tilde{f}(c_{I_{0,-}}) - (\mathbf{L}+1) \vert X - c_{I_{0,-}} \vert \leqslant \frac{1}{2} + C \rho^{N}  , X \in I_{0,-})  \\
& + \mathbb{P}(\tilde{f}(c_{I_{0,+}}) - (\mathbf{L}+1) \vert X - c_{I_{0,+}} \vert\leqslant \frac{1}{2} + C \rho^{N} ,X \in I_{0,+}) ,
\end{align*}

where $c_{I_{0,-}} = \frac{\inf\{x \in I_{0,-}\}+\sup\{x \in I_{0,-}\}}{2}$ (and  similarly for $c_{I_{0,+}}$) is the center point of $I_{0,-}$.
When $I_{0}=I_{N+1}$, we have similarly

\begin{align*}
\mathbb{P}(\tilde{f}(X) \leqslant \frac{1}{2} + C \rho^{N})  \leqslant &  \mathbb{P}(X \leqslant \frac{1}{2} + C \rho^{N} , X \notin I_{N+1})\\
& +  \mathbb{P}(\tilde{f}(c_{I_{N+1}}) - (\mathbf{L}+1) \vert X - c_{I_{N+1}} \vert  \leqslant \frac{1}{2} + C \rho^{N}  ,X \in I_{N+1})  .
\end{align*}

Now, we study each terms appearing in the upper bounds of $\mathbb{P}(\tilde{f}(X) \leqslant \frac{1}{2} + C \rho^{N}) $ above. For $I_{j} \in \tilde{D}^{N}$, we have,  

\begin{align*}
\mathbb{P}(X \leqslant \frac{1}{2} + C \rho^{N},  X \notin I_{j}) =& \mathbb{P}(X \leqslant \frac{1}{2} + C \rho^{N})  - \mathbb{P}(X \leqslant \frac{1}{2} + C \rho^{N},  X \in   I_{j})  \\
\leqslant & \frac{1}{2} + C \rho^{N}  - \mathbb{P}(X \in  [\frac{1}{2}(1-\boldsymbol{\rho}^{N}),\frac{1}{2}(1+\min(\boldsymbol{\rho}^{N},2 C \rho^{N}))]) \mathds{1}_{j=N+1} \\
& -\mathbb{P}(X \in  I_{j,-})  \mathds{1}_{j \ne\boldsymbol{\rho} N+1} \\
=& \frac{1}{2} + C \rho^{N}  - \frac{1}{2} (\min(\boldsymbol{\rho}^{N},2 C \rho^{N})  +\boldsymbol{\rho}^{N})\mathds{1}_{j=N+1} -\boldsymbol{\rho}^{j-1}\frac{1-\boldsymbol{\rho}}{2}  \mathds{1}_{j \ne\boldsymbol{\rho} N+1}  .
\end{align*}

Moreover, for $I_{0}=I_{j} \in \{ I_{1}, \ldots I_{N} \}$
\begin{align*}
\tilde{f}(c_{I_{j,-}}) = & c_{I_{j,-}}+  \mathbf{L} \inf_{y  \notin \ I_{j,-}} \vert c_{I_{j,-}} -y  \vert  =\frac{1}{2} -\boldsymbol{\rho}^{j-1}\frac{1+\boldsymbol{\rho}}{4}  +  \mathbf{L} \boldsymbol{\rho}^{j-1} \frac{1-\boldsymbol{\rho}}{4} = \frac{1}{2} +\boldsymbol{\rho}^{j-1}\frac{\mathbf{L}(1-\boldsymbol{\rho})-1-\boldsymbol{\rho}}{4} ,
\end{align*}
and similarly $\tilde{f}(c_{I_{j,+}})  = \frac{1}{2} +\boldsymbol{\rho}^{j-1}\frac{\mathbf{L}(1-\boldsymbol{\rho})+1+\boldsymbol{\rho}}{4} $. In addition, when $I_{0}=I_{N+1}$, we also compute (recall that $c_{I_{N+1}}= \frac{1}{2}$)
\begin{align*}
\tilde{f}(c_{I_{N+1}}) = &  \frac{1}{2} +  \mathbf{L} \inf_{y  \notin I_{N+1}} \vert \frac{1}{2} -y  \vert  =\frac{1}{2} +  \mathbf{L}  \frac{1}{2}\boldsymbol{\rho}^{N} .
\end{align*}

Exploiting those computations, it follows that for $I_{0}=I_{j} \in \{I_{1},\ldots,I_{N}\}$,

\begin{align*}
\mathbb{P}(\tilde{f}(c_{I_{j,-}}) - (\mathbf{L}+1) \vert X - c_{I_{j,-}} \vert_{\mathbb{R}^{d}} \leqslant \frac{1}{2} +& C \rho^{N}  \vert X \in I_{j,-})   \\
\leqslant   & \mathbb{P}( \vert X - c_{I_{j,-}} \vert_{\mathbb{R}^{d}} \geqslant \frac{1}{\mathbf{L}+1} (\boldsymbol{\rho}^{j-1}\frac{\mathbf{L}(1-\boldsymbol{\rho})-1-\boldsymbol{\rho}}{4} - C \rho^{N})  \vert X \in I_{j,-}) \\
= &(1- \frac{2}{\boldsymbol{\rho}^{j-1}(1-\boldsymbol{\rho})} \frac{2}{\mathbf{L}+1} (\boldsymbol{\rho}^{j-1} \frac{\mathbf{L}(1-\boldsymbol{\rho})-1-\boldsymbol{\rho}}{4} - C \rho^{N})_{+})_{+}.
\end{align*}
Then, using the Bayes formula yields,
\begin{align*}
\mathbb{P}(\tilde{f}(c_{I_{j,-}}) - & (\mathbf{L}+1) \vert X - c_{I_{j,-}} \vert_{\mathbb{R}^{d}} \leqslant \frac{1}{2} + C \rho^{N} ,X \in I_{j,-}) \\  
\leqslant & \frac{1}{2}\boldsymbol{\rho}^{j-1}(1-\boldsymbol{\rho})(1- \frac{2}{\boldsymbol{\rho}^{j-1}(1-\boldsymbol{\rho})} \frac{2}{\mathbf{L}+1} (\frac{\mathbf{L}(1-\boldsymbol{\rho})-1-\boldsymbol{\rho}}{4} \boldsymbol{\rho}^{j-1}- C \rho^{N})_{+})_{+} \\
\leqslant & \frac{1}{2}\boldsymbol{\rho}^{j-1}(1-\boldsymbol{\rho})(1-\frac{\mathbf{L}(1-\boldsymbol{\rho})-1-\boldsymbol{\rho}}{(1-\boldsymbol{\rho})(\mathbf{L}+1)}+\frac{4 C \rho^{N}}{(1-\boldsymbol{\rho})(\mathbf{L}+1)}\boldsymbol{\rho}^{-j+1} ) \\
=& \boldsymbol{\rho}^{j-1}(\frac{1+2 C \rho^{N} \boldsymbol{\rho}^{-j+1} }{\mathbf{L}+1} ).
\end{align*}
By a similar approach, we also derive
\begin{align*}
\mathbb{P}(\tilde{f}(c_{I_{j,+}}) -& (\mathbf{L}+1) \vert X - c_{I_{j,+}} \vert_{\mathbb{R}^{d}} \leqslant \frac{1}{2} + C \rho^{N}  ,X \in I_{j,+})   \\
\leqslant & \frac{1}{2}\boldsymbol{\rho}^{j-1}(1-\boldsymbol{\rho})(1- \frac{2}{\boldsymbol{\rho}^{j-1}(1-\boldsymbol{\rho})} \frac{2}{\mathbf{L}+1} (\frac{\mathbf{L}(1-\boldsymbol{\rho})+1+\boldsymbol{\rho}}{4} \boldsymbol{\rho}^{j-1}- C \rho^{N})_{+})_{+} \\
\leqslant &\boldsymbol{\rho}^{j-1}(\frac{-\boldsymbol{\rho}+2 C \rho^{N} \boldsymbol{\rho}^{-j+1} }{\mathbf{L}+1} )_{+}.
\end{align*}
In the same way,  when $I_{0}=I_{N+1}$,

\begin{align*}
\mathbb{P}(\tilde{f}(c_{I_{N+1}}) -  (\mathbf{L}+1) \vert X - & c_{I_{N+1}} \vert_{\mathbb{R}^{d}} \leqslant \frac{1}{2} + C \rho^{N} , X \in I_{N+1} )  \\
 \leqslant   & \mathbb{P}( \vert X - c_{I_{N+1}} \vert_{\mathbb{R}^{d}} \geqslant \frac{1}{\mathbf{L}+1} (\frac{\mathbf{L}}{2}\boldsymbol{\rho}^{N} - C \rho^{N}),  X \in I_{N+1}) \\
 =& (\boldsymbol{\rho}^{N}-\frac{2}{\mathbf{L}+1} (\frac{\mathbf{L}}{2}\boldsymbol{\rho}^{N} - C \rho^{N})  _{+})_{+} \\
  \leqslant &\frac{1}{\mathbf{L}+1}   \boldsymbol{\rho}^{N} +  \frac{2}{\mathbf{L}+1}  C \rho^{N}.
\end{align*}

We conclude that for $I_{0}=I_{j} \in \{I_{1},\ldots,I_{N}\}$,

\begin{align*}
\mathbb{P}(\tilde{f}(X) \leqslant \frac{1}{2} + C \rho^{N})  \leqslant &   \frac{1}{2} + C \rho^{N} -\boldsymbol{\rho}^{j-1}\frac{1-\boldsymbol{\rho}}{2}  + \boldsymbol{\rho}^{j-1}\frac{1}{\mathbf{L}+1} +  \frac{4}{\mathbf{L}+1} C \rho^{N}\\
\leqslant & \frac{1}{2}+\frac{\mathbf{L}+5}{\mathbf{L}+1}C \rho^{N} -\boldsymbol{\rho}^{j-1} \frac{(\mathbf{L}+1)(1-\boldsymbol{\rho})-2}{2 (\mathbf{L}+1)},
\end{align*}
and for $I _{0}=I_{N+1}$,
\begin{align*}
\mathbb{P}(\tilde{f}(X) \leqslant \frac{1}{2} + C \rho^{N})  \leqslant &   \frac{1}{2} + C \rho^{N} -\frac{1}{2}\boldsymbol{\rho}^{N}  +   \frac{1}{\mathbf{L}+1}  \boldsymbol{\rho}^{N} +  \frac{2}{\mathbf{L}+1}  C \rho^{N}\\
=& \frac{1}{2}+\frac{\mathbf{L}+3}{\mathbf{L}+1}C \rho^{N} - \frac{\mathbf{L}-1}{2(\mathbf{L}+1)} \boldsymbol{\rho}^{N} .
\end{align*}

It now remains to guarantee that both upper bounds we just derived are strictly lower than $\frac{1}{2}$. We fix $\boldsymbol{\rho} \in (0,1)$ and  $\mathbf{L} \in (\frac{1+\boldsymbol{\rho}}{1-\boldsymbol{\rho}},+\infty)$ so that $ \frac{\mathbf{L}-1}{2(\mathbf{L}+1)} >0$ and $(\mathbf{L}+1)(1-\boldsymbol{\rho})-2>0$. Therefore, if $C$ satisfies
\begin{align*}
C < \min ( \frac{(\mathbf{L}+1)(1-\boldsymbol{\rho})-2}{2\boldsymbol{\rho} (\mathbf{L}+5)} ,\frac{\mathbf{L}-1}{2(\mathbf{L}+3) }  ) ,
\end{align*}
and $\rho \leqslant \boldsymbol{\rho}$, then
\begin{align*}
\mathbb{P}(\tilde{f}(X) \leqslant \frac{1}{2} + C \rho^{N})  <& \frac{1}{2},
\end{align*}
which is the expected conclusion (choose for instance $\mathbf{L}=4$, $\rho=\boldsymbol{\rho}=\frac{1}{2}$ and $C<\frac{1}{18}$).

\end{proof}

\section{Numerical illustration}
\label{Numerical_illustration}
To conclude this article, we propose a numerical illustration of Theorem \ref{th:conv_gen} and Theorem \ref{th:conv_gen_unknown}. In both cases, we propose an application for $d=1$ and $d=2$. When $d=1$, we expect to observe an exponential convergence and when $d=2$, we expect a polynomial convergence of order $\frac{1}{N}$. Let us present our examples. \\
\subsection{Exponential convergence}
First for $d=1$, we consider $X \sim \mathcal{N} (\frac{1}{5}, \frac{1}{25} ) \mathds{1}_{[0,1]}$ a Gaussian distribution with mean $\frac{1}{5}$, variance $\frac{1}{25} $ and restricted to the interval $[0,1]$. The function $f$ is defined on $x \in [0,1]$ by
\begin{align*}
f(x) = 0.8x-0.3+\exp(-11.534x^{1.95})+\exp(-2(x-0.9)^{2}).
\end{align*}
This function is studied in \cite{Bernard_Cohen_Guyader_Malrieu_2022} for the approximation of failure probabilities. In particular, we have $L_{f} \approx 1.61$.  We fix $\alpha=0.999$. In this case $q(f,X) \approx 1.3503$. \\

In Figures \ref{fig:fig1} (A) and (B), we represent respectively
\begin{align}
\label{eq:error_log_numer_connu}
\ln(\vert q_{\alpha}(f,X) -q^{N}_{\alpha}(f,X) \vert ),
\end{align}
and
\begin{align}
\label{eq:error_log_numer_inconnu}
\ln(\vert q_{\alpha}(f,X) -q^{\diamond,N}_{\alpha}(f,X) \vert ),
\end{align}
 resulting from Algorithm \ref{algo:gen_Lip_connu} and \ref{algo:gen_Lip_inconnu} $w.r.t.$ $N$. Those quantities appear in blue while the red line represents the linear approximation with slope respectively given by $\ln(0.8453)$ for Figure \ref{fig:fig1} (A) and by $\ln(0.9988)$ for Figure \ref{fig:fig1} (B). In particular, we can infer the numerical approximation for $\rho$ given by $\rho \approx 0.8453$ for Algorithm \ref{algo:gen_Lip_connu} and $\rho \approx 0.9988$ for Algorithm \ref{algo:gen_Lip_inconnu} \\
 
 \begin{figure}[ht]
    \centering
   
    \begin{subfigure}{0.45\textwidth}
        \centering
        \includegraphics[width=\textwidth]{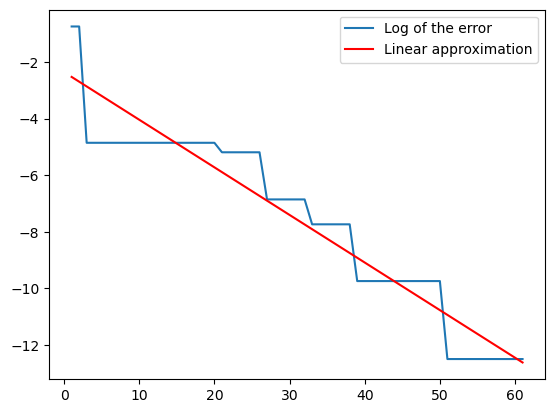}
        \caption{}
        \label{Fig:d1_known}
    \end{subfigure}
    \hfill
    \begin{subfigure}{0.45\textwidth}
        \centering
        \includegraphics[width=\textwidth]{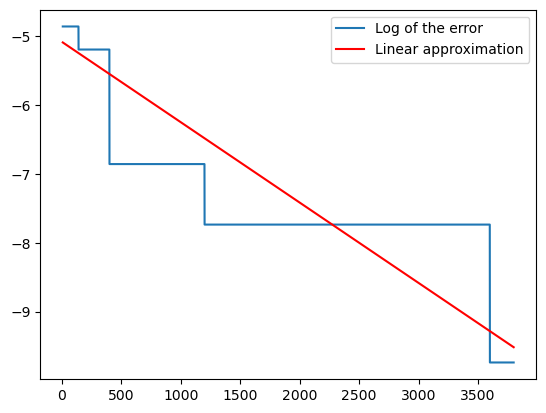}
        \caption{}
        \label{Fig:d1_unknown}
    \end{subfigure}

    \caption{Logarithm of the error of estimation of Algorithm \ref{algo:gen_Lip_connu} (see Figure (A)) and \ref{algo:gen_Lip_inconnu} (see Figure (B)) $w.r.t.$ the number of observations.}
    \label{fig:fig1}
\end{figure}
 
 The drawings are consistent with the results from Theorem \ref{th:conv_gen} and Theorem \ref{th:conv_gen_unknown}, as both exhibit a linearly decreasing behavior. This confirms the exponential nature of convergence for both algorithms. However, in this practical case,  Algorithm \ref{algo:gen_Lip_connu} shows faster numerical convergence, $i.e.$, it has a smaller $\rho$. This is expected, as in the unknown Lipschitz case, Algorithm \ref{algo:gen_Lip_inconnu} tests multiple Lipschitz constant candidates and apply a similar implementation to Algorithm \ref{algo:gen_Lip_connu}  for each candidate.  The budget allocated for each candidate is only a fraction of the total global budget, $N$, allocated across all candidates. Additionally, the piecewise constant behavior of the error $w.r.t.$ $N$ for both algorithms is also expected. This occurs because our algorithms require a certain budget to subdivide more deeply the space $[0,1]$ $i.e.$, to progress from subdivision with depth $k \in \mathbb{N}$ to depth $k+1$. Until that budget is reached, the algorithms will continue returning the same result.\\

\subsection{Polynomial convergence}
In a second step, we propose an example when $d=2$. In this case, we consider $X=(X_{1},X_{2})$ where $X_{1}$ and $X_{2}$ are independent and identically distributed under $\mathcal{U}_{[0,1]}$, the uniform distribution on $[0,1]$. The function $f$ is defined on $x \in [0,1]^{2}$ by $f(x)=x_{1}+x_{2}$.  In this toy example, $f$ is Lipschitz with Lipschitz constant $L_{f}=\sqrt{2}$, $f(X)$ follows the Irwin Hall distribution of degree 2 and we know that $q_{\alpha}(f,X)=2-\sqrt{2(1-\alpha)}$. As before, we fix $\alpha=0.999$ and we have $q_{\alpha}(f,X)\approx 1.9553$.\\

In Figures \ref{fig:fig2} (A) and (B), we represent respectively the quantities (\ref{eq:error_log_numer_connu}) and (\ref{eq:error_log_numer_inconnu})
 resulting from Algorithm \ref{algo:gen_Lip_connu} and \ref{algo:gen_Lip_inconnu} $w.r.t.$ $\ln(N)$. Those quantities appear in blue while the red line represents the linear approximation with slope respectively given by -1.4781 for Figure \ref{fig:fig2} (A) and by -0.7440 for Figure \ref{fig:fig2} (B).  \\

 \begin{figure}[ht]
    \centering
   
    \begin{subfigure}{0.45\textwidth}
        \centering
        \includegraphics[width=\textwidth]{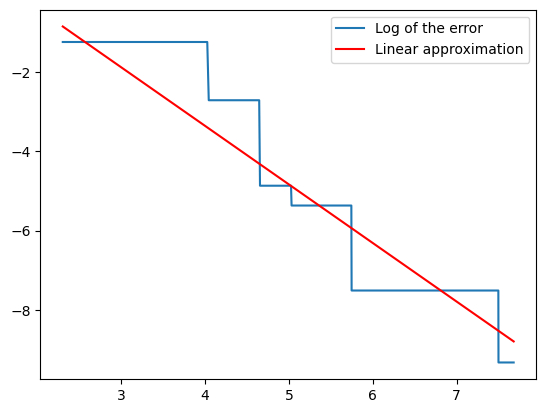}
        \caption{}
        \label{Fig:d2_known}
    \end{subfigure}
    \hfill
    \begin{subfigure}[b]{0.45\textwidth}
        \centering
        \includegraphics[width=\textwidth]{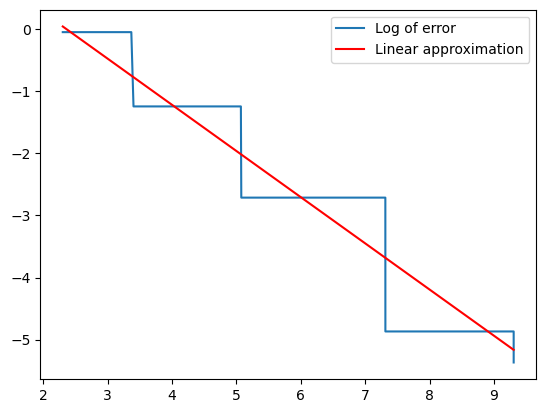}
        \caption{}
        \label{Fig:d2_unknown}
    \end{subfigure}

    \caption{Logarithm of the error of estimation of Algorithm \ref{algo:gen_Lip_connu} (see Figure (A)) and \ref{algo:gen_Lip_inconnu} (see Figure (B)) $w.r.t.$ the Logarithm of the number of observations.}
    \label{fig:fig2}
\end{figure}
The expected linear decreasing behavior with slope $1=\frac{1}{d-1}$ is thus observed (to some extent) aligning with the result from Theorem \ref{th:conv_gen} and Theorem \ref{th:conv_gen_unknown}. Similarly to the case when $d=1$, we find that while both algorithms exhibit similar convergence behavior, the numerical approximations converge more quickly when the Lipschitz constant of $f$ is known, i.e. when Algorithm \ref{algo:gen_Lip_connu} is used.

\bibliography{Biblio}

\def\cprime{$'$} \def\cprime{$'$}
\begin{thebibliography}{10}

\bibitem{Bellini_Klar_Muller_Rosazza_2014}
F.~Bellini, B.~Klar, A.~Müller, and E.~Rosazza~Gianin.
\newblock Generalized quantiles as risk measures.
\newblock {\em Insurance: Mathematics and Economics}, 54:41--48, 2014.

\bibitem{Bernard_Cohen_Guyader_Malrieu_2022}
L.~Bernard, A.~Cohen, A.~Guyader, and F.~Malrieu.
\newblock Recursive {Estimation} of a {Failure} {Probability} for a {Lipschitz}
  {Function}.
\newblock {\em The SMAI Journal of computational mathematics}, 8:75--97, 2022.

\bibitem{Cannamela_Garnier_Iooss_2008}
C.~Cannamela, J.~Garnier, and B.~Iooss.
\newblock Controlled stratification for quantile estimation.
\newblock {\em The Annals of Applied Statistics}, 2(4):1554--1580, 2008.

\bibitem{Cohen_Devore_Petrova_Wojtaszczyk_2014}
A.~Cohen, R.~Devore, G.~Petrova, and P.~Wojtaszczyk.
\newblock Finding the minimum of a function.
\newblock {\em Methods and Applications of Analysis}, 20(4):365--382, 2014.

\bibitem{Cerou_Guyader_2016}
F.~Cérou and A.~Guyader.
\newblock Fluctuation analysis of adaptive multilevel splitting.
\newblock {\em The Annals of Applied Probability}, 26(6):3319--3380, 2016.

\bibitem{Cerou_Guyader_Rousset_2019}
F.~Cérou, A.~Guyader, and M.~Rousset.
\newblock {Adaptive multilevel splitting: Historical perspective and recent
  results}.
\newblock {\em Chaos: An Interdisciplinary Journal of Nonlinear Science},
  29(4):043108, 04 2019.

\bibitem{DelMoral_Furon_Guyader_2012}
P.~Del~Moral, T.~Furon, and A.~Guyader.
\newblock Sequential monte carlo for rare event estimation.
\newblock {\em Statistics and Computing}, 22(3):795--808, 2012.

\bibitem{Edgeworth_1888}
F.~Y. Edgeworth.
\newblock The mathematical theory of banking.
\newblock {\em Journal of the Royal Statistical Society}, 51(1):113--127, 1888.

\bibitem{Egloff_Leippold_2010}
D.~Egloff and M.~Leippold.
\newblock Quantile estimation with adaptive importance sampling.
\newblock {\em Annals of Statistics}, 38:1244--1278, 2010.

\bibitem{Evans_Gariepy_2015}
L.C. Evans and R.F. Gariepy.
\newblock {\em Measure Theory and Fine Properties of Functions, Revised
  Edition}.
\newblock Textbooks in Mathematics. CRC Press, 2015.

\bibitem{Glynn_2011}
P.~W. Glynn.
\newblock Importance sampling for monte carlo estimation of quantiles.
\newblock 2011.

\bibitem{Goffinet_Wallach_1996}
B.~Goffinet and D.~Wallach.
\newblock Optimised importance sampling quantile estimation.
\newblock {\em Biometrika}, 83(4):791--800, 1996.

\bibitem{Gomes_Pestana_2007}
M.I. Gomes and D.~Pestana.
\newblock A sturdy reduced-bias extreme quantile (var) estimator.
\newblock {\em Journal of the American Statistical Association},
  102(477):280--292, 2007.

\bibitem{Greenwald_Khanna_2016}
M.~B. Greenwald and S.~Khanna.
\newblock {\em Quantiles and Equi-depth Histograms over Streams}, pages 45--86.
\newblock Springer Berlin Heidelberg, Berlin, Heidelberg, 2016.

\bibitem{Guyader_Hengartner_Matzner_2011}
A.~Guyader, N~Hengartner, and E.~Matzner-Lober.
\newblock Simulation and estimation of extreme quantiles and extreme
  probabilities.
\newblock {\em Applied Mathematics \& Optimization}, 64(2):171--196, 2011.

\bibitem{Hallin_Paindaveine_Siman_2010}
M.~Hallin, D.~Paindaveine, and M~Siman.
\newblock Multivariate quantiles and multiple-output regression quantiles: from
  l1 optimization to halfspace depth.
\newblock {\em Annals of Statistics}, 38:635--669, 2010.

\bibitem{Hesterberg_Nelson_1998}
T.~C. Hesterberg and B.~L. Nelson.
\newblock Control variates for probability and quantile estimation.
\newblock {\em Management Science}, 44(9):1295--1312, 1998.

\bibitem{Hsu_Nelson_1990}
J.~C. Hsu and B.~L. Nelson.
\newblock Control variates for quantile estimation.
\newblock {\em Management Science}, 36(7):835--851, 1990.

\bibitem{Koenker_2005}
R.~Koenker.
\newblock Quantile regression.
\newblock {\em Econometric Society Monographs, Cambridge University Press,
  Cambridge}, 2005.

\bibitem{Lee_epedelenlioglu_Spania_Muniraju_2020}
J.~Lee, C.~Tepedelenlioğlu, A.~Spania, and G.~Muniraju.
\newblock Distributed quantiles estimation of sensor network measurements.
\newblock 7:38--61, 2020.

\bibitem{Byon_Ko_Lam_Pan_Young_2020}
Q.~Pan, E.~Byon, Young~M. Ko, and H.~Lam.
\newblock Adaptive importance sampling for extreme quantile estimation with
  stochastic black box computer models.
\newblock {\em Naval Research Logistics (NRL)}, 67(7):524--547, 2020.

\bibitem{Tambwekar_Maiya_Dhavala_Saha_2021}
A.~Tambwekar, A.~Maiya, S.~Dhavala, and S.~Saha.
\newblock Estimation and applications of quantiles in deep binary
  classification.
\newblock {\em IEEE Transactions on Artificial Intelligence}, 3:275--286, 2021.

\bibitem{Ullah_Luo_Adebayo_Sunday_Kartal_2023}
S.~Ullah, R.~Luo, T.~Sunday~Adebayo, and M.~T. Kartal.
\newblock Dynamics between environmental taxes and ecological sustainability:
  evidence from top-seven green economies by novel quantile approaches.
\newblock {\em Sustainable Development}, 31(2):825--839, 2023.

\bibitem{Yang_Zhao_2018}
H.~Yang and Y.~Zhao.
\newblock Smoothed jackknife empirical likelihood for the one-sample difference
  of quantiles.
\newblock {\em Comput. Stat. Data Anal.}, 120:58--69, 2018.

\end{thebibliography}
\bibliographystyle{plain}

%
\end{document}